\documentclass[final, nomarks]{dmtcs-episciences}

\usepackage[utf8]{inputenc}
\usepackage{subfigure}

\usepackage{amsmath}
\usepackage{amsthm}
\usepackage{todonotes}
\usepackage{tikz}
\usetikzlibrary{decorations.pathreplacing}

\usepackage{oubraces}
\usepackage{bm}

\def\A{\mathcal A}
\def\uu{\mathbf{u}}
\def\vv{\mathbf{v}}

\usepackage{xcolor}

\newcommand{\red}[1]{\textcolor{red}{#1}}

\newcommand{\teckaa}[2]{#2\raisebox{-1.5ex}{\hspace{-0.5ex}\textcolor{#1}{\small$\bullet $}}}
\newcommand{\tec}{\raisebox{-0.2ex}{\small$\bullet $}}

\theoremstyle{definition}
\newtheorem{definition}{Definition}
\newtheorem{corollary}[definition]{Corollary}
\newtheorem{remark}[definition]{Remark}
\newtheorem{example}[definition]{Example}

\theoremstyle{plain}
\newtheorem{theorem}[definition]{Theorem}
\newtheorem{proposition}[definition]{Proposition}
\newtheorem{lemma}[definition]{Lemma}
\newtheorem{observation}[definition]{Observation}

\author[L. Dvo\v r\'akov\'a and M. Moravcov\'a]{Lubom\'ira Dvo\v r\'akov\'a \and Martina Moravcov\'a\thanks{Supported by Czech Technical University in Prague, through the project SGS26/160/OHK4/3T/14}}
\title{Attractors of sequences coding $\beta$-integers}
\affiliation{
Czech Technical University in Prague, Czech Republic}

\begin{document}

\publicationdata{vol. 28:3}{2026}{20}{10.46298/dmtcs.16852}{2025-11-04; 2025-11-04; 2026-07-21}{2026-10-28}

\maketitle

\begin{abstract}
In this paper, we describe minimal string attractors of prefixes of simple Parry sequences. These sequences form a~coding of distances between consecutive $\beta$-integers in numeration systems with a~real base $\beta$. 
Simple Parry sequences have been recently studied by Gheeraert, Romana, and Stipulanti from this point of view and attractors of prefixes have been described. However, the authors of that study themselves had doubts about their minimality and conjectured that attractors of alphabet size should be sufficient. We confirm their conjecture. Moreover, we provide minimal attractors of prefixes of some particular form of binary non-simple Parry sequences.
\end{abstract}

\section{Introduction}
Recently, the concept of string attractors has been a subject of study in theoretical computer science. It plays an important role, especially in the field of data compression. This concept was introduced and first studied by Kempa and Prezza~\cite{KempaPrezza2018}: a \emph{string attractor} of a finite word $w=w_0w_1\cdots w_{n-1}$, where $w_i$ are letters, is a subset $\Gamma$ of $\{0,1,\dots, n-1\}$ such that each non-empty factor of $w$ has an occurrence containing an element of $\Gamma$.  
In general, however, the problem of finding an attractor of minimal size of a word is NP-complete. Therefore, it is natural to study attractors in the field of combinatorics on words for specific significant word classes, where the properties of such words are extensively exploited and the problem thus becomes solvable. 

To date, minimal attractors have been found for factors / prefixes / particular prefixes of several classes of sequences (infinite words)~\cite{Mantaci2021, Dv2023, Shallit2021, Kutsukake2020, Dolce2023, Kociumaka2021, DvHe2024}. The relation between new string attractor-based complexity functions and other well-known combinatorial complexity functions was studied in~\cite{Ca2024}. String attractors in bi-infinite sequences were discussed in~\cite{BeGhMe2024} and effective algorithms for checking attractors and producing attractors of reasonable size were described in~\cite{BeCrRo2025}.

Recently, Gheeraert, Romana, and Stipulanti~\cite{GhRoSt2024} have described attractors of prefixes of simple Parry sequences. Slightly more general fixed points were studied there. The minimal attractors have been found in the case of simple Parry sequences with affine factor complexity. (For the description of parameters guaranteeing affine factor complexity see~\cite{BeMaPe2007}). 

Parry sequences are closely connected to non-standard numeration systems where instead of an integer base one considers a real base $\beta > 1$. Every non-negative real number $x$ may be expressed using the base $\beta$ in the form
\begin{equation*}
    x = \sum_{i = -\infty}^{k} x_i\beta^i, \quad \text{where $x_i \in \mathbb{N}$, $k \in \mathbb{Z}$ and $x_k \not= 0$}\,.
\end{equation*}
For negative numbers $x$ the minus sign is employed.
Such a representation may be obtained using the greedy algorithm; we then speak about the {\em $\beta$-expansion}. 
Numbers whose $\beta$-expansion has a vanishing fractional part, i.e., numbers in the form $\pm \sum_{i=0}^k x_i\beta^i$, where $k \in \mathbb{N}$, are called {\em $\beta$-integers}. The set of $\beta$-integers $\mathbb Z_\beta$ is therefore an analogue of integers for non-integer bases. If there are only finitely many different distances between consecutive $\beta$-integers, we may code them with letters. The sequences obtained are called {\em simple / non-simple Parry sequences} and $\beta$ is called a {\em simple / non-simple Parry number}.
Parry sequences have been studied by many authors from different points of view~\cite{AmFrMaPe2006, BaKlPe2011, BaMa2009, BeMaPe2007, DoMa2015, KlPe2009, MaPe2011, Tu2015}.

In this article, we follow up on the work~\cite{GhRoSt2024}. We describe minimal attractors of prefixes of simple Parry sequences, which confirm the hypothesis from~\cite{GhRoSt2024}, where the authors believed that attractors of alphabet size exist. 
Moreover, we describe attractors of prefixes of some particular form of binary non-simple Parry sequences, which is a partial answer to another open question from the same paper.

This paper starts with preliminaries, where we define attractors and mention their basic properties. In Section~\ref{sec:Parry}, we introduce $\beta$-integers and numeration systems and describe their relation to simple / non-simple Parry sequences. We recall how to obtain Parry sequences as fixed points of morphism. 

In Section~\ref{sec:simpleParry}, we describe attractors of prefixes of simple Parry sequences. We consider separately two cases. Under some additional assumptions, the attractors of prefixes form a subset of $\{|\varphi^n(0)|-1 \ : \ n \in \mathbb N\}$, where the considered Parry sequence is the fixed point of $\varphi$. See Theorem~\ref{veta: SP_atraktory}. When we relax the additional conditions, attractors of alphabet size may still be found; however, they do not form a subset of $\{|\varphi^n(0)|-1 \ : \ n \in \mathbb N\}$ anymore. See Theorem~\ref{veta: SP_atraktory2}. We illustrate the results on various examples. 

In Section~\ref{sec:nonsimpleParry}, attractors of prefixes of some particular form of binary non-simple Parry sequences are provided. 
At the end of the paper, we mention some open questions. 

\section{Preliminaries}
\label{Section_Preliminaries}
An {\em alphabet} $\A$ is a finite set of symbols, called {\em letters}. A {\em word of length $n$} over $\A$ is a finite sequence $u = u_0u_1 \cdots u_{n-1}$, where $u_i \in \A$. The length of $u$ is denoted by $|u|$. The set $\A^*$ consists of all finite words over $\A$. This set with the operation of concatenation forms a monoid, the neutral element is the {\em empty word} $\varepsilon$. We denote $\A^+ = \A^* \setminus \{\varepsilon\}$.  An infinite sequence $\uu = u_0u_1u_2 \cdots$, where $u_i \in \A$, is called an {\em infinite word} or simply a {\em sequence} over $\A$. Sequences will be denoted by bold letters.

Let $u \in \A^*$, $u = xyz$ for some $x,y,z \in \A^*$. The word $x$ is called a~{\em prefix} of $u$, $z$ a~{\em suffix} of $u$ and $y$ a~{\em factor} of $u$.

Consider $\uu$ a sequence over $\A$, $\uu = u_0u_1u_2 \cdots$. A~word $y$ such that $y = u_iu_{i+1}u_{i+2} \cdots u_{j-1}$ for some $i, j \in \mathbb{N}$, $i \leq j$, is called a~{\em factor} of $\uu$. If $i=j$, then $y = \varepsilon$. If $i = 0$, then $y$ is called a~{\em prefix} of~$\uu$. The set $\{i, i+1, i+2, \ldots, j-1\}$ is said to be an {\em occurrence} of $y$ in the sequence $\uu$.~\footnote{It is more common to call only $i$ an occurrence of $y$ in $\uu$, but in the context of attractors, the modified definition is more suitable.} An occurrence of a~factor in a~finite word is defined analogously. 

Let $u \in \A^*$. Denote $u^k = uu \cdots u$ (i.e., the word $u$ repeated $k$ times), where $k \in \mathbb{N}$, and call it the $k$-th {\em power} of $u$. Similarly, $u^{\omega} = uuu \cdots$ denotes an infinite concatenation of $u$. A~sequence $\uu$ over $\A$ is called {\em eventually periodic} if $\uu = vw^{\omega}$ for some $v \in \A^*$ and $w \in \A^+$. In particular, $\uu$ is {\em periodic} if $v = \varepsilon$. Furthermore, $\uu$ is {\em aperiodic} if $\uu$ is not eventually periodic. 

A~factor $w$ of a sequence $\uu$ over $\A$ is called a~{\em left special factor} if $aw$, $bw$ are factors of $\uu$ for two distinct letters $a, b \in \A$. We say that $\uu$ is {\em closed under reversal} if for each factor $w = w_0w_1 \cdots w_{n-1}$ of $\uu$ its {\em reversal} $w_{n-1} \cdots w_1w_0$ is a factor of $\uu$, too. A~binary sequence $\uu$ is called {\em Sturmian} if $\uu$ is closed under reversal and $\uu$ contains exactly one left special factor of every length. 

A~mapping $\varphi:\A^* \rightarrow \A^*$ satisfying for all $u, v \in \A^*$
\begin{equation*}
    \varphi(uv) = \varphi(u)\varphi(v)
\end{equation*}
is called a~{\em morphism}. Let $\uu$ be a~sequence over $\mathcal A$, $\uu = u_0u_1u_2 \cdots$. The morphism may be applied also to sequences 
\begin{equation*}
    \varphi(\uu) = \varphi(u_0u_1u_2 \cdots) = \varphi(u_0)\varphi(u_1)\varphi(u_2) \cdots
\end{equation*}
A~sequence $\uu$ is a~{\em fixed point of the morphism} $\varphi$ if $\varphi(\uu)=\uu$.

Let $u, v \in \A^*$. We say that the word $u$ is a~{\em power} of the word $v$ if $u = v^kv'$, where $k \in \mathbb{N}$ and $v'$ is a prefix of $v$. For instance,  the word $barbar=(bar)^2$ is a square of $bar$ or the word $salsa$ is a power of $sal$. 

\begin{definition}
Let $\uu$, $\vv$ be two sequences over $\{0, 1, \ldots, d\}$ for some $d \in \mathbb{N}$, $\uu = u_0u_1u_2 \cdots$ and $\vv = v_0v_1v_2 \cdots$. We say that $\uu$ is {\em lexicographically smaller (greater)} than $\vv$, we write $\uu \prec_{\text{lex}} \vv$ ($\uu \succ_{\text{lex}} \vv$), if for the smallest index $i\in \mathbb{N}$ such that $u_i \neq v_i$ we have $u_i < v_i$ ($u_i>v_i$). 
\end{definition}

\subsection{Attractors}

A~\emph{(string) attractor} of a finite word $w=w_0w_1\cdots w_{n-1}$, where $w_i$ are letters, is a subset $\Gamma$ of $\{0,1,\dots, n-1\}$ such that each non-empty factor of $w$ has an occurrence in $w$ containing an element of $\Gamma$. 
If $i \in \Gamma$ and the word $f$ has an occurrence in $w$ containing $i$, we say that $f$ crosses $i$ and we also say that $f$ crosses the attractor $\Gamma$. An attractor of the word $w$ with the minimal number of elements is called a~{\em minimal attractor} of the word $w$. For example, $\Gamma = \{0,1,5\}$ is an attractor of the word $w = \red{{a}}\red{n}ana\red{s}$ (the letters corresponding to the positions of $\Gamma$ are written in red). The factors $an$ and $ana$ cross the positions $0$ and $1$, the factor $na$ crosses the position $1$, and all of them thus cross the attractor $\Gamma$. This attractor is minimal since every attractor of $w$ necessarily contains occurrences of all distinct letters of $w$.

Let us state a simple observation concerning the attractors of powers of words.
\begin{observation}\label{lem: atraktory_faktor}
Let $x$ be a power of a word $z$ and $x = z^nz'$, where $n \in \mathbb{N}, n\geq 1$, and $z'$ is a prefix of $z$. Let $f$ be a factor of $x$. 
\begin{itemize}
\item If $f$ has an occurrence in $x$ crossing $i|z|-1$ for some $i \in \mathbb{N}$, $1\leq i < n$, then $f$ has an occurrence in $x$ crossing $j|z|-1$ for each $j \in \mathbb{N}$, $1\leq j < n$. 
\item If $f$ has an occurrence in $x$ crossing $|z|-1$, but $f$ has no occurrence crossing $n|z|-1$, then $f$ is a~suffix of $z^{n-1}z''$, where $z''$ is a prefix of $z$ and $|z|>|z''|>|z'|$.
\end{itemize}
\end{observation}

The following useful lemma is taken from the paper~\cite{GhRoSt2024} (Proposition 8).
It also easily follows from the above observation. 

\begin{lemma}\label{lem: atraktory_mocnina}
Let $x, y$ be powers of a word $z$ and $|z| \leq |x| \leq |y|$. If $\Gamma$ is an attractor of $x$, then $\Gamma \cup \{|z| - 1\}$ is an attractor of $y$.
\end{lemma}

A useful straightforward consequence is summarized in the following corollary.
\begin{corollary}\label{lem: atraktory_mocnina_jednodušší}
Let $x$ be a power of a word $z$ and $|z| \leq |x|$. If $\Gamma$ is an attractor of $z$, then $\Gamma \cup \{|z| - 1\}$ is an attractor of $x$. 
\end{corollary}

\section{Parry sequences and non-standard numeration systems}\label{sec:Parry}
As mentioned in the introduction, Parry sequences code distances between consecutive $\beta$-integers, where $\beta$-integers generalize the notion of integers to numeration systems with a real base $\beta>1$. In this section, we describe this concept in a more formal way. We draw information from~\cite{Lothaire} (Chapter 7 Numeration systems).

\subsection{$\beta$-expansion}

\begin{definition}
Consider a base $\beta \in \mathbb{R}$, $\beta > 1$, and $x \in \mathbb{R}$, $x \geq 0$. If
\begin{equation*}
    x = \sum_{i = -\infty}^k x_i\beta^i, \quad  \text{where $k \in \mathbb{Z}$, $x_k \neq 0$ and $x_i \in \mathbb{N}$  for each $i \in \mathbb{Z}$, $i \leq k$}\,,
\end{equation*}
then $x_k \cdots x_1x_0\tec  x_{-1}x_{-2} \cdots$ is called a~$\beta$-{\em representation} of $x$. In particular, a~$\beta$-representation of $x$ obtained by the {\em greedy algorithm} is called the $\beta$-{\em expansion} of $x$ and is denoted by $\langle x \rangle_{\beta}$.
\end{definition}


\begin{example}
For $\beta=\frac{1+\sqrt{5}}{2}$, it holds $\beta^2=\beta+1$. We have 
\begin{equation*}
    \langle 3 \rangle_{\beta}=100\tec  01, \quad \langle \sqrt 5\rangle_\beta=10\tec  1 \quad \text{or} \quad \langle{\beta^2}/{2}\rangle_\beta=1\tec  (001)^{\omega}\,.
\end{equation*}
\end{example}

Thanks to the greedy algorithm, the lexicographic order of $\beta$-expansions corresponds to the order of non-negative real numbers. 

\begin{lemma}\label{lem:radix_velikost}
Consider a real base $\beta > 1$ and $x, y \in \mathbb{R}$, $x \geq 0$, $y \geq 0$, with $\langle x \rangle_{\beta} = x_k \cdots x_1x_0\tec  x_{-1}x_{-2} \cdots$ and $\langle y \rangle_{\beta} = y_\ell \cdots y_1y_0\tec  y_{-1}y_{-2} \cdots$. Then $x<y$ if and only if $k < \ell$ or $k = \ell$ and $\langle x \rangle_{\beta} \prec_{\text{lex}} \langle y \rangle_{\beta}$.
\end{lemma}

\begin{lemma}\label{lem:interval01}
Consider a real base $\beta > 1$ and a real number $x \geq 0$, $\langle x \rangle_{\beta} = x_k \cdots x_1x_0\tec  x_{-1}x_{-2} \cdots$. Then  $\langle \frac{x}{\beta} \rangle_{\beta} = x_k \cdots x_1\tec  x_0x_{-1}x_{-2} \cdots$.
\end{lemma}
Thanks to Lemma~\ref{lem:interval01}, it suffices to know the $\beta$-expansions of numbers in the interval $[0,1)$ to get the $\beta$-expansions of all real numbers.

\subsection{Rényi expansion of unity}

The $\beta$-expansion of numbers from the interval $[0,1)$ may be computed using the transformation $T_\beta:
[0,1]\rightarrow [0,1)$ defined as
\begin{equation}\label{T_beta}
T_{\beta}(x)=\{\beta x\}= \beta x - \lfloor \beta x \rfloor\,.
\end{equation}

For each $x \in [0,1)$, it holds that ${\langle x
\rangle}_{\beta}=0 \tec  x_{-1} x_{-2}\cdots$ if and only if
\begin{equation}\label{beta_transform}
x_{-i}=\lfloor \beta T_{\beta}^{i-1}(x) \rfloor\,.
\end{equation}
For $x=1$, the formula~(\ref{beta_transform}) does not provide the $\beta$-expansion of $1$ since ${\langle 1
\rangle}_{\beta}=1\tec $. Nevertheless, it gives us a useful tool,
the Rényi expansion of unity (defined in~\cite{Re1957}). 
\begin{definition}\label{def:Renyi_rozvoj1}
Let $\beta \in \mathbb R, \ \beta>1$. Then the {\em Rényi expansion of unity}
in the base $\beta$ is defined as
\begin{equation}\label{d_beta_1}
d_{\beta}(1)=t_1t_2t_3\cdots, \quad \mbox{where} \ t_i:=\lfloor
\beta T_{\beta}^{i-1}(1)\rfloor\,. 
\end{equation}

\end{definition}

Since $t_1=\lfloor \beta \rfloor$, we have $t_1\geq 1$. On the one hand, every number $\beta>1$ is uniquely given by its R\'enyi expansion of unity. On the other hand,  not every sequence of non-negative integers is equal to $d_{\beta}(1)$ for some $\beta$.
Parry solved this problem~\cite{Pa1960}: The sequence of numbers
$(t_i)_{i \geq 1}$, $t_i \in \mathbb N$, is the Rényi expansion of unity for some number $\beta>1$ if and only if it satisfies the lexicographic condition
\begin{equation}\label{ParryRenyi}
    t_jt_{j+1}t_{j+2}\cdots \prec_{\text{lex}} t_1t_2t_3\cdots \quad \mbox{for each $j > 1$\,.}
\end{equation}
In particular, it implies that the Rényi expansion of unity is never periodic. Parry moreover proved that the Rényi expansion of unity 
enables us to decide whether a~$\beta$-representation of a positive number $x$ is its $\beta$-expansion. For this purpose, we introduce the {\em infinite Rényi expansion of unity} (it is the lexicographically greatest infinite $\beta$-representation of unity).
\begin{equation}\label{infinite_exp}
d^{*}_{\beta}(1)= \left\{\begin{array}{ll} d_{\beta}(1) & \hbox{if}
\ d_{\beta}(1) \ \text{is infinite}\,,\\
(t_1 t_2\cdots t_{m-1}(t_m -1))^{\omega}
 & \hbox{if}
\ d_{\beta}(1)=t_1\cdots t_m\,, \ \text{where} \ t_m \not =
0\,.
\end{array}\right.
\end{equation}

\begin{proposition}[Parry condition] \label{Parry_expansions} Consider a real base $\beta > 1$ and a real number $x\geq 0$.
Let $d^{*}_{\beta}(1)$ be the infinite Rényi expansion of unity and let $x_k \cdots x_1x_0\tec  x_{-1}x_{-2} \cdots$ be a~$\beta$-representation of $x$. Then  $x_k \cdots x_1x_0\tec  x_{-1}x_{-2} \cdots$ is the $\beta$-expansion of $x$ if and only if
\begin{equation}\label{ParryCondition}
x_ix_{i-1}\cdots \prec_{\text{lex}} d^{*}_{\beta}(1) \quad \mbox{ for all }\ i\leq k\,.
\end{equation}
\end{proposition}

\begin{example}
For $\beta=\frac{1+\sqrt{5}}{2}$, the Rényi expansion of unity $d_\beta(1)=11$. Thus, $d^{*}_\beta(1)=(10)^{\omega}$. Applying the Parry condition, we can see that every sequence of coefficients in~$\{0,1\}$, which does not end in $(10)^{\omega}$ and does not contain the block $11$, is the $\beta$-expansion of a non-negative real number. 
\end{example}

\begin{definition}
A real number $\beta>1$, which has an eventually periodic Rényi expansion of unity, is called a~{\em Parry number}. If the expansion $d_\beta(1)$ is finite, then $\beta$ is a~{\em simple Parry number}. If the expansion $d_\beta(1)$ is infinite, then $\beta$ is a~{\em non-simple Parry number}.
\end{definition}

\subsection{$\beta$-integers}
Consider a real number $\beta >1$. Real numbers $x$ whose $\beta$-expansion has a vanishing fractional part are called {\em $\beta$-integers} and their set is denoted by $\mathbb Z_\beta$. Formally written
\begin{equation*}
    {\mathbb Z}_{\beta}:=\{x \in \mathbb R \ : \ {\langle |x| \rangle}
    _{\beta}=x_k x_{k-1}\cdots x_0 \tec  \}\,.
\end{equation*}

According to Lemma~\ref{lem:radix_velikost}, the lexicographic order of $\beta$-expansions corresponds to the order of numbers with respect to the size.  
Therefore, there exists an increasing sequence $(b_n)_{n
=0}^{\infty}$ such that 
\begin{equation}\label{b_n}
\{b_n \ : \ n \in \mathbb
N\}=\mathbb Z_\beta \cap [0,\infty)\,.
\end{equation}

Since $\mathbb Z_\beta=\mathbb Z$ for any integer $\beta>1$, the distance between consecutive elements of $\mathbb Z_{\beta}$ is always one. This situation radically changes if $\beta \notin \mathbb N$. In this case, the number of different distances between neighboring elements of $\mathbb Z_\beta$ is at least two. 

Thurston~\cite{Th1989} showed that the distances between neighbors in $\mathbb Z_{\beta}$ form the set 
$\mbox{$\{\Delta_{k} \ : \ k \in \mathbb N \}$}$, where
\begin{equation}\label{distance_beta}
    \Delta_{k}:=\sum_{i=1}^{\infty}\frac{t_{i+k}}{{\beta}^{i}} \quad \hbox{for} \ k \in \mathbb N\,.
\end{equation}
Obviously, the set $\{\Delta_{k}\ : \ k \in \mathbb N \}$ is finite if and only if $d_{\beta}(1)$ is an eventually periodic sequence, i.e., $\beta$ is a Parry number.
If the number of distances of consecutive elements in $\mathbb Z_\beta$ is finite, we may code the same distances with the same letters. In such a way, we obtain a sequence $\uu_\beta$ encoding $\mathbb Z_\beta \cap [0,\infty)$, as illustrated in Figure~\ref{tau}.

\begin{example} \label{example_u_beta}
Consider again $\beta=\frac{1+\sqrt{5}}{2}$, i.e., $d_\beta(1)=11$ and $d_\beta^*(1)=(10)^{\omega}$. According to the formula~(\ref{distance_beta}), we observe that the distances between neighboring $\beta$-integers attain two values: $\Delta_0=1$
and $\Delta_1=\frac{1}{\beta}$. 
When coding the distances $\Delta_0 \rightarrow 0$ and $\Delta_1 \rightarrow 1$, we get the famous Fibonacci sequence. A~prefix is written in Figure~\ref{tau}. 

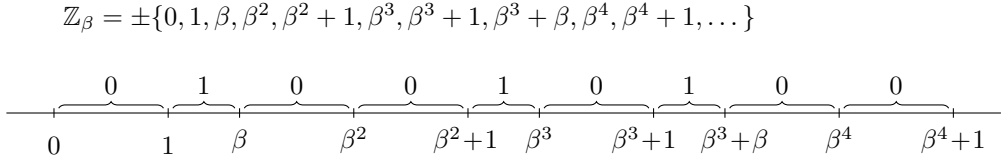
\begin{figure}[ht]
\centering
\begin{tikzpicture}[x=0.9pt,y=0.9pt]
\node[anchor=west] at (20,58) {$\mathbb Z_\beta = \pm\{0, 1, \beta, \beta^2, \beta^2+1, \beta^3, \beta^3+1, \beta^3+\beta, \beta^4, \beta^4+1, \dots\}$};
\draw (0,18) -- (418,18);
\foreach \zacatek/\konec/\popis in {
    22/66/0, 
    70/96/1, 
    100/144/0, 
    148/192/0, 
    196/222/1,
    226/270/0, 
    274/300/1, 
    304/348/0, 
    352/396/0
    }{
    \draw[decorate,decoration={brace}] (\zacatek,20) -- (\konec,20) node[midway, above, yshift=2pt] {$\popis$};
    }
\foreach \bod/\popis in {
    20/$0$, 
    68/$1$, 
    98/$\beta$, 
    146/$\beta^2$, 
    194/$\beta^2\!+\!1$, 
    224/$\beta^3$,
    272/$\beta^3\!+\!1$, 
    302/$\beta^3\!+\!\beta$, 
    350/$\beta^4$, 
    398/$\beta^4\!+\!1$
    }{
    \draw (\bod,16) -- (\bod,20);
    }
\foreach \bod/\popis in {
    20/$0$, 
    68/$1$, 
    98/$\beta$, 
    146/$\beta^2$, 
    194-1/$\beta^2\!+\!1$, 
    224/$\beta^3$,
    272-4/$\beta^3\!+\!1$, 
    302+4/$\beta^3\!+\!\beta$, 
    350/$\beta^4$, 
    398/$\beta^4\!+\!1$
    }{
    \node[anchor=south] at (\bod,-3) { \popis};
    }
\end{tikzpicture}
\caption{Illustration of coding of distances in $\mathbb Z_\beta$ for $\beta=\frac{1+\sqrt{5}}{2}$} \label{tau}
\end{figure}
\end{example}

\subsection{Morphisms and Parry numbers}\label{Parry_substitution} 
Fabre~\cite{Fa1995} noticed that the sequences $\uu_\beta$ coding non-negative $\beta$-integers for Parry bases $\beta$ are fixed points of morphisms.

More precisely, if $\beta$ is a~simple Parry number, i.e., $d_\beta(1) =t_1t_2\cdots
t_m$ for $m\in \mathbb N, m\geq 2$, then $\uu_\beta$ is the fixed point of the morphism $\varphi$ defined over the alphabet $\{0, 1, \dots,
m-1\}$ in the following way
\begin{equation}\label{subst1general}
\begin{array}{rcl}
\varphi(0)&=&0^{t_1}1\,, \\
\varphi(1)&=&0^{t_2}2\,, \\
&\vdots &\\
\varphi(m-2)&=&0^{t_{m-1}}(m-1)\,,\\
\varphi(m-1)&=&0^{t_m}\,.
\end{array}
\end{equation}
The sequence $\uu_\beta$ is called a~{\em simple Parry sequence}.

\vspace{0.3cm}
Similarly, let $\beta$ be a~non-simple Parry number, i.e., let $m,r\in \mathbb N, m\geq 1, r\geq 1$, be minimal such that $d_\beta(1)=t_1t_2\cdots t_m(t_{m+1}\cdots t_{m+r})^{\omega}$, then $\uu_\beta$ is the fixed point of the morphism $\varphi$ defined over the alphabet $\{0, 1, \dots, m+r-1\}$ as follows 
\begin{equation}\label{subst2general}
\begin{array}{rcl}
\varphi(0)&=&0^{t_1}1\,, \\
\varphi(1)&=&0^{t_2}2\,, \\
&\vdots& \\
\varphi(m-1)&=&0^{t_m}m\,, \\
&\vdots& \\
\varphi(m+r-2)&=&0^{t_{m+r-1}}(m+r-1)\,,\\
\varphi(m+r-1)&=&0^{t_{m+r}}m\,.
\end{array}
\end{equation}
The sequence $\uu_\beta$ is called in this case a~{\em non-simple Parry sequence}.

Let us recall an essential relation between a~$\beta$-integer $b_n$ and its coding by a~prefix of the associated infinite word $\uu_\beta$.
\begin{proposition}[Fabre~\cite{Fa1995}]\label{prop: Fabre}
Let $\uu_\beta$ be the sequence associated with a~Parry number $\beta$ and let $\varphi$ be the
associated morphism. Then for every $\beta$-integer $b_n \in \mathbb Z_\beta \cap [0,\infty)$
it holds that $\langle b_n \rangle_\beta = x_{k-1}x_{k-2}\cdots x_1 x_0\tec $ if and
only if $\varphi^{k-1}(0^{x_{k-1}})\varphi^{k-2}(0^{x_{k-2}})\cdots \varphi(0^{x_1})0^{x_0}$ is a~prefix of $\uu_\beta$ of length $n$.
\end{proposition}

\section{Attractors of prefixes of simple Parry sequences}\label{sec:simpleParry}

Gheeraert, Romana, and Stipulanti~\cite{GhRoSt2024} described for simple Parry sequences~\footnote{They worked with more general sequences -- fixed points of morphisms from~\eqref{subst1general} with non-negative integer coefficients $t_1,\dots, t_m$ and $t_1, t_m\geq1$.} attractors of prefixes whose size was by one larger than the alphabet size, see Theorem~\ref{thm:puvodni}. They conjectured that attractors of alphabet size exist.
The aim of this section is to prove their conjecture. 

They also asked under which condition the attractors of prefixes form a~subset of $\{|\varphi^n(0)|-1\ : \ n \in \mathbb N\}$. We partially answer this question, too.

Let us recall the definition of simple Parry sequences in the form of fixed points of morphisms, the assumptions on parameters follow from the properties of the Rényi expansion of unity~\eqref{ParryRenyi}.
\begin{definition}\label{def:simple-Parry} Let $m\in \mathbb N, m \geq 2$. A \emph{simple Parry sequence} $\uu$ is a~fixed point of the morphism $\varphi: \{0,1,\dots, m-1\}^*\to \{0,1,\dots, m-1\}^*$ defined as
\begin{equation*}
\begin{array}{rcl}
    \varphi(0)&=&0^{t_1}1\,, \\
    \varphi(1)&=&0^{t_2}2\,, \\
    &\vdots& \\
    \varphi(m-2)&=&0^{t_{m-1}}(m-1)\,, \\
    \varphi(m-1)&=&0^{t_m}\,,
\end{array}
\end{equation*}
where $t_1, t_2, \ldots, t_m \in \mathbb{N}$, $t_1 \geq 1$, $t_m \geq 1$, and moreover
\begin{equation*}
    t_it_{i+1} \cdots t_m0^{\omega} \prec_{\text{lex}} t_1t_2 \cdots t_m0^{\omega} \quad \text{for each $i \in \{2, \ldots, m\}$}\,.
\end{equation*}
\end{definition}

We will denote $u_n=\varphi^n(0)$ and $U_n=|u_n|$ for $n \in \mathbb N$.
We set $u_n=\varepsilon$ and $U_n=0$ for $n<0$.

Clearly, $u_{n+1} = \varphi(u_n)$ and $u_n$ is a prefix of $u_{n+1}$.
\begin{remark}
For $m=2$, it is known that $\uu$ is Sturmian if and only if $t_2=1$.
Arnoux-Rauzy sequences among simple Parry sequences are exactly the ones with $t_1=t_2=\dots=t_{m-1}$ and $t_m=1$.
Attractors of prefixes of Sturmian sequences~\cite{Mantaci2021} and Arnoux-Rauzy sequences~\cite{Dv2023} are known.
\end{remark}

\begin{example}\label{pr: SP_priklad}
For $m=3$ and $t_1=2$, $t_2=t_3=1$, 
the morphism takes on the following form

\begin{equation*}
\begin{array}{rclcl}
\varphi(0) &=& 001\,, \\
\varphi(1) &=& 02\,, \\
\varphi(2) &=& 0\,,
\end{array}
\end{equation*}
a~few first prefixes $u_n$ of $\uu$ look as follows
\begin{equation*}
\begin{array}{lcl}
u_0 &=& 0\,, \\
u_1 &=& 001\,, \\
u_2 &=& 00100102\,, \\
u_3 &=& 00100102001001020010\,, \\
u_4 &=& 001001020010010200100010010200100102001000100102001\,, \\
u_5 &=& 00100102001001020010001001020010010200100010010200100100102001001020010001\\
&&00102001001020010001001020010010010200100102001000100102\,.
\end{array}
\end{equation*}

For $m=4$ and $t_1=2$, $t_2=1$, $t_3=2$ and $t_4=1$, the morphism is defined as

\begin{equation*}
\begin{array}{rclcl}
\varphi(0) &=& 001\,, \\
\varphi(1) &=& 02\,, \\
\varphi(2) &=& 003\,,\\
\varphi(3) &=& 0\,,
\end{array}
\end{equation*}
and the shortest prefixes $u_n$ of $\uu$ are
\begin{equation*}
\begin{array}{rcl}
u_0 &=& 0\,, \\
u_1 &=& 001\,, \\
u_2 &=& 00100102\,, \\
u_3 &=& 0010010200100102001003\,, \\
u_4 &=& 00100102001001020010030010010200100102001003001001020010010\,.
\end{array}
\end{equation*}
\end{example}

We start with several handy lemmas. Lemma~\ref{lem: SP_rekurze}, resp. Lemma~\ref{lem: SP_mocnina} can be found in~\cite{GhRoSt2024} as Proposition~4, resp. Theorem~22. We add the proof of Lemma~\ref{lem: SP_mocnina} since it was proved there using a more general setting. 
\begin{lemma}\label{lem: SP_rekurze}
For each $n \in \mathbb{N}$, $1\leq n \leq m-1$, it holds
\begin{equation*}
    u_n = u_{n-1}^{t_1}u_{n-2}^{t_2} \cdots u_0^{t_n}n\,.
\end{equation*}
For each $n \in \mathbb{N}$, $n \geq m$, it holds
\begin{equation*}
    u_n = u_{n-1}^{t_1}u_{n-2}^{t_2} \cdots u_{n-m}^{t_m}\,.
\end{equation*}
\end{lemma}






\begin{example}
Let us illustrate Lemma~\ref{lem: SP_rekurze} on the prefixes from Example~\ref{pr: SP_priklad}, where $m=3$ and $t_1=2$, $t_2=t_3=1$. The prefixes of $\uu$ satisfy
\begin{equation*}
\begin{array}{lcl}
u_0 &=& 0\,, \\
u_1 &=& \underbrace{0}_{u_0}\underbrace{0}_{u_0}1 = u_0^21\,, \\
u_2 &=& \underbrace{001}_{u_1}\underbrace{001}_{u_1}\underbrace{0}_{u_0}2 = u_1^2u_02\,, \\
u_3 &=& \underbrace{00100102}_{u_2}\underbrace{00100102}_{u_2}\underbrace{001}_{u_1}\underbrace{0}_{u_0} = u_2^2u_1u_0\,, \\
u_4 &=& \underbrace{00100102001001020010}_{u_3}\underbrace{00100102001001020010}_{u_3}\underbrace{00100102}_{u_2}\underbrace{001}_{u_1} = u_3^2u_2u_1\,, \\
u_5 &=& \underbrace{001001020010010200100010010200100102001000100102001}_{u_4}\\
&&\underbrace{001001020010010200100010010200100102001000100102001}_{u_4}\\
&&\underbrace{00100102001001020010}_{u_3}\underbrace{00100102}_{u_2} = u_4^2u_3u_2\,. \\

\end{array}
\end{equation*}

\end{example}

\begin{lemma}\label{lem: SP_prefixy}
For each $n \in \mathbb{N}$, $n\geq 1$, 
\begin{equation*}
    u_{n-1}^{k_1}u_{n-2}^{k_2} \cdots u_0^{k_n} \quad \text{is a prefix of $u_n$}
\end{equation*}
if $k_1, k_2, \ldots, k_n \in \mathbb{N}$ satisfy $k_ik_{i+1} \cdots k_n0^{\omega} \prec_{\text{lex}} t_1t_2 \cdots t_m0^{\omega}$ for all $i \in \{1, \ldots, n\}$. 
\end{lemma}

\begin{proof}
Using the Parry condition from Proposition~\ref{Parry_expansions} for $d^*_\beta(1)=(t_1t_2\cdots t_{m-1}(t_m-1))^{\omega}$, there is a~$\beta$-integer $b_N$ with $\langle b_N\rangle_\beta=k_1 k_2 \cdots k_n\tec$.

Then, applying Proposition~\ref{prop: Fabre}, the word $u_{n-1}^{k_1}u_{n-2}^{k_2} \cdots u_0^{k_n}$ is a prefix of $\uu$. 
Finally, by Lemma~\ref{lem:radix_velikost}, since $k_1 k_2 \cdots k_n$ is shorter than $10^n$, the word $u_{n-1}^{k_1}u_{n-2}^{k_2} \cdots u_0^{k_n}$ is a prefix of $u_n$. \end{proof}

\begin{example}
Let us illustrate Lemma~\ref{lem: SP_prefixy} on the prefixes of $\uu$ from Example~\ref{pr: SP_priklad}, where $m=3$ and $t_1=2$, $t_2=t_3=1$. 
For instance for $k_1 = 1$, $k_2 = 2$, $k_3 = 1$, $k_4=0$ and $k_5=1$, the prefix $u_5$ of $\uu$ looks as follows
\begin{equation*}
\begin{array}{lcl}
u_5 &=& \underbrace{001001020010010200100010010200100102001000100102001}_{u_4}\\
&&\underbrace{00100102001001020010}_{u_3}\underbrace{00100102001001020010}_{u_3}\underbrace{00100102}_{u_2}\underbrace{0}_{u_0}\\
&&010010010200100102001000100102\\
&=&u_4u_3^2u_2u_0010010010200100102001000100102\,.
\end{array}
\end{equation*}

\end{example}

\begin{lemma}\label{lem: SP_mocnina}
The word $u_{n+1}$ without the last letter is a power of $u_n$ for all $n \in \mathbb{N}$. Moreover, the word $u_{n+1}$ is a power of $u_n$ for all $n \in \mathbb{N}$, $n \geq m-1$.
\end{lemma}
\begin{proof}
For $n = 0$, we have $u_1 = u_0^{t_1}1$.
For each $n \in \mathbb{N}$, $1\leq n < m-1$, Lemma~\ref{lem: SP_rekurze} implies that
\begin{equation*}
    u_{n+1} = u_n^{t_1}u_{n-1}^{t_2}u_{n-2}^{t_3} \cdots u_0^{t_{n+1}}(n+1) = u_n^{t_1}u_n'(n+1)\,,
\end{equation*}
where $u_n' = u_{n-1}^{t_2}u_{n-2}^{t_3} \cdots u_0^{t_{n+1}}$ is a~prefix of $u_n$ by Lemma~\ref{lem: SP_prefixy} since $t_{i}t_{i+1} \cdots t_{n+1}0^{\omega} \prec_{\text{lex}} t_1t_2 \cdots t_m0^{\omega}$ for all $i \in \{2, \ldots, n+1\}$. Thus $u_{n+1}$ without the last letter is a power of $u_n$.

For each $n \in \mathbb{N}$, $n \geq m-1$, it holds using Lemma~\ref{lem: SP_rekurze}
\begin{equation*}
    u_{n+1} = u_n^{t_1}u_{n-1}^{t_2}u_{n-2}^{t_3} \cdots u_{n-m+1}^{t_m} = u_n^{t_1}u_n'\,,
\end{equation*}
where $u_n' = u_{n-1}^{t_2}u_{n-2}^{t_3} \cdots u_{n-m+1}^{t_m}$ is a~prefix of $u_n$ by Lemma~\ref{lem: SP_prefixy} since $t_it_{i+1} \cdots t_m0^{\omega} \prec_{\text{lex}} t_1t_2 \cdots  t_m0^{\omega}$ for all $i \in \{2, \ldots, m\}$. Consequently, $u_{n+1}$ is a power of $u_n$.
\end{proof}

\begin{example}
Let us illustrate Lemma~\ref{lem: SP_mocnina} again on the prefixes of $\uu$ from Example~\ref{pr: SP_priklad}, where $m=3$ and $t_1=2$, $t_2=t_3=1$. The prefixes of $\uu$ satisfy
\begin{equation*}
\begin{array}{lcl}
u_0 &=& 0\,, \\
u_1 &=& \underbrace{0}_{u_0}\underbrace{0}_{u_0}1 = u_0^21\,, \\
u_2 &=& \underbrace{001}_{u_1}\underbrace{001}_{u_1}\underbrace{0}_{u_1'}2 = u_1^2u_1'2\,, \\
u_3 &=& \underbrace{00100102}_{u_2}\underbrace{00100102}_{u_2}\underbrace{0010}_{u_2'} = u_2^2u_2'\,, \\
u_4 &=& \underbrace{00100102001001020010}_{u_3}\underbrace{00100102001001020010}_{u_3}\underbrace{00100102001}_{u_3'} = u_3^2u_3'\,, \\
u_5 &=& \underbrace{001001020010010200100010010200100102001000100102001}_{u_4}\\
&&\underbrace{001001020010010200100010010200100102001000100102001}_{u_4}\\
&&\underbrace{0010010200100102001000100102}_{u_4'} = u_4^2u_4'\,.
\end{array}
\end{equation*}
\end{example}

Let us turn our attention to the attractors of prefixes of $m$-ary simple Parry sequences. First, we summarize known results from~\cite{GhRoSt2024}.
We keep the following notation: $\Gamma_{-1}=\emptyset$ and
\begin{equation}\label{eq:Gamma}
    \Gamma_n = \begin{cases}
        \{U_0-1, U_1-1, \ldots, U_n-1\} & \text{for } n \in \mathbb{N}, n \leq m-1\,; \\
        \{U_{n-m+1}-1, U_{n-m+2}-1, \ldots, U_n-1\} & \text{for}\  n \in \mathbb{N}, n \geq m\,.
    \end{cases}
\end{equation}
The attractors of prefixes of $m$-ary simple Parry sequences of size $m+1$ may be deduced using Theorem~10 from~\cite{GhRoSt2024}. The authors used the notation $Q_n$ for the length of the longest prefix of $\uu$ that is a~power of $u_n$ and
\begin{equation*}
    P_n=\begin{cases} 
        U_n & \quad \text{for } n \in \mathbb{N}, n \leq m-1\,;\\
        U_n+U_{n-m+1}-U_{n-m}-1 & \quad \text{for}\  n \in \mathbb{N}, n \geq m\,.
    \end{cases}
\end{equation*}
Obviously, $U_n\leq P_n<U_{n+1}$.

Using Lemma~\ref{lem: SP_mocnina}, the assumption of Theorem~10 from~\cite{GhRoSt2024} that every prefix of length $U_{n+1}-1$ is a~power of $u_n$ is met. Consequently, Theorem~10 from~\cite{GhRoSt2024} applied to simple Parry sequences takes the following form.
\newpage
\begin{theorem}\label{thm:puvodni}
Let $\uu$ be a simple Parry sequence from Definition~\ref{def:simple-Parry}. 
For all $n\in \mathbb N$, 
\begin{enumerate}
\item every prefix of length $\ell \in [U_n, Q_n]$ has the attractor $\Gamma_{n-1}\cup \{U_n-1\}$;
\item every prefix of length $\ell \in [P_n, Q_n]$ has the attractor $\Gamma_{n}$.
\end{enumerate}
\end{theorem}

Since $Q_n\geq U_{n+1}-1$ for $n< m-1$ and $Q_n\geq U_{n+1}$ for $n\geq m-1$ by Lemma~\ref{lem: SP_mocnina}, we immediately obtain the following corollary showing that each prefix of an $m$-ary simple Parry sequence has an attractor of size at most $m+1$.
\begin{corollary}\label{coro: thm10}
Let $\uu$ be a simple Parry sequence from Definition~\ref{def:simple-Parry}. 
For all $n\in \mathbb N$, 
\begin{enumerate}
\item every prefix of length $\ell \in [U_n, U_{n+1}-1]$ has the attractor $\Gamma_{n}$ for $n\leq m-1$;
\item every prefix of length $\ell \in [U_n, U_{n+1}]$ has the attractor $\Gamma_{n-1}\cup \{U_n-1\}$ for $n\geq m$;
\item every prefix of length $\ell \in [P_n, U_{n+1}]$ has the attractor $\Gamma_{n}$ for $n\geq m-1$;
\item every prefix of length $\ell \in [U_{n+1}, Q_{n}]$ has the attractor $\Gamma_{n}$ for $n\geq m-1$.
\end{enumerate}
\end{corollary}
Let us underline that the only prefixes for which the attractors of alphabet size are not known have their lengths in the interval $[Q_{n-1}, P_n]$ for $n\geq m$.

The authors of~\cite{GhRoSt2024} were also interested in the question of whether the minimal attractors of prefixes are always subsets of $\{U_n-1 \ : \ n\in \mathbb N\}$. Let us call such attractors {\em canonical}. Observing Corollary~\ref{coro: thm10}, all attractors of prefixes found in~\cite{GhRoSt2024} were canonical.
Let us recall one more result from~\cite{GhRoSt2024}, where the authors proved that for simple Parry sequences with affine factor complexity, the minimal attractors of prefixes are canonical and their minimality follows from the fact that their size equals the number of distinct letters in the prefix. 
\begin{theorem}[\cite{GhRoSt2024}]\label{thm: affine} 
Let $\uu$ be a simple Parry sequence from Definition~\ref{def:simple-Parry} with affine factor complexity, i.e., satisfying the following conditions:
\begin{enumerate}
\item $t_m=1$;
\item if there exists a word $v\neq \varepsilon$ such that $v$ is a proper prefix and a proper suffix of $t_1\cdots t_{m-1}$, then $t_1\cdots t_{m-1}=w^k$ for some word $w$ and $k\in \mathbb N, k\geq 2$.
\end{enumerate}
Then the prefixes of $\uu$ have the following attractors:
\begin{itemize} 
\item For each $n \in \mathbb{N}$, $n \leq m-1$, the prefix of $\uu$ of length $\ell \in [U_n, U_{n+1}-1]$ has the attractor $\Gamma_n$. 
\item For each $n \in \mathbb{N}$, $n \geq m$, the prefix of $\uu$ of length $\ell \in [U_n, P_n]$ has the attractor $\Gamma_{n-1}$.
\item For each $n \in \mathbb{N}$, $n \geq m-1$, the prefix of $\uu$ of length $\ell \in [P_n, U_{n+1}]$ has the attractor $\Gamma_n$.
\end{itemize}

\end{theorem}

In the sequel, in order to obtain new results, it turns out to be useful to work with prefixes other than those of length $P_n$. Let us introduce them.
To enable comparison, let us write the explicit form of the prefix of length $P_n$ for $n\geq m$. 
Denote by $p_n$ the following prefix of $\uu$
\begin{equation}\label{eq: p}
p_n= u_{n}u_{n-m}^{t_1-1}u_{n-m-1}^{t_2} \cdots u_{n-2m+1}^{t_m}\,.
\end{equation}
Then for $m\leq n < 2m-1$, the length of $p_n$ equals $P_n$, and for $n\geq 2m-1$, the length of $p_n$ equals $P_n+1$.

\begin{lemma}\label{lem: z}
Let $\uu$ be a simple Parry sequence from Definition~\ref{def:simple-Parry}. For $n\in \mathbb N,\ n\geq m$, denote 
\begin{equation}\label{eq: z}
\begin{array}{rcl}
z_n &=& u_{n}u_{n-m}^{t_1-t_m}u_{n-m-1}^{t_2} \cdots u_{n-2m+1}^{t_m}\,,\\
&&\\
s_n&=&u_{n}u_{n-m}^{t_1-t_m+1}u_{n-m-1}^{t_2} \cdots u_{n-2m+1}^{t_m}\,,
\end{array}
\end{equation}
and denote $Z_n=|z_n|$ and $S_n=|s_n|$. Then 
\begin{itemize}
\item both $z_n$ and $s_n$ are prefixes of $\uu$;
\item $Z_n=U_n+U_{n-m+1}-t_mU_{n-m}$ \ for $n\geq 2m-1$;
\item $S_n=Z_n+U_{n-m}$;
\item $U_n\leq Z_n< S_n \leq U_{n+1}$;
\item $Z_n\leq P_n$ \ for $t_m>1$.
\end{itemize}

\end{lemma}
\begin{proof}
The words $z_n$ and $s_n$ are prefixes of $\uu$ since $u_{n-m}^{t_1-t_m}u_{n-m-1}^{t_2}\cdots u_{n-2m+1}^{t_m}$, resp. \\ $u_{n-m}^{t_1-t_m+1}u_{n-m-1}^{t_2}\cdots u_{n-2m+1}^{t_m}$ is a prefix of $u_{n-m+1}$ by Lemma~\ref{lem: SP_prefixy} and $u_n u_{n-m+1}$ is a prefix of $u_{n+1}$ by the same lemma. The statements on lengths follow from Lemma~\ref{lem: SP_rekurze}.

\end{proof}

Using the prefixes $z_n$ and $s_n$, we can deduce the following statement. 
\begin{proposition}\label{prop: z}
Let $\uu$ be a simple Parry sequence from Definition~\ref{def:simple-Parry}. Let $n\in \mathbb N, \ n\geq m$.
\begin{enumerate}
\item If $t_1>t_m,$ then every prefix of length $\ell \in [Z_n, U_{n+1}]$ has the attractor $\Gamma_{n}$.
\item If $t_1=t_m,$ then 
\begin{itemize}
\item every prefix of length $\ell \in [Z_n, S_n]$ has the attractor 
\begin{equation*}
    \Gamma_{n-1}\cup \{U_n-(t_m-1)U_{n-m}-1\}\setminus \{U_{n-m}-1\}\,;
\end{equation*}
\item every prefix of length $\ell \in [S_n, U_{n+1}]$ has the attractor $\Gamma_{n}$.
\end{itemize}
\end{enumerate}
Moreover, in all cases, the attractors are minimal and they differ by at most one position from a canonical attractor. 
\end{proposition}

\begin{proof}
Using Item 3 of Corollary~\ref{coro: thm10}, we observe that $u_{n}$ has the attractor $\Gamma_{n-1}$ for all $n\geq m$. Since $u_{n+1}$ is a power of $u_n$ and $U_n\leq Z_n <U_{n+1}$, the prefix $z_n$ has, by Lemma~\ref{lem: atraktory_mocnina}, the attractor $\Gamma_{n-1}\cup \{U_n-1\}=\{U_{n-m}-1, U_{n-m+1}-1, \ldots, U_n-1\}$. 
\begin{enumerate}
\item For $t_1>t_m$, let us explain that every factor of $z_n$ crossing $U_{n-m}-1$, but not $U_{n-m+1}-1$, has also an occurrence containing $U_{n}-1$. By Lemma~\ref{lem: SP_rekurze}, the word $z_n$ has the following form (we illustrate the form for $n\geq 2m-1$, the reader is invited to draw the prefix for $n<2m-1$; it does not influence the subsequent arguments)

\begin{equation}\label{eq:z_n1}
\begin{split}
    z_n &= u_{n}u_{n-m}^{t_1-t_m}u_{n-m-1}^{t_2} \cdots u_{n-2m+1}^{t_m}  \\
    &= u_{n-1}^{t_1}u_{n-2}^{t_2} \cdots u_{n-m+1}^{t_{m-1}}u_{n-m}^{t_m}u_{n-m}^{t_1-t_m}u_{n-m-1}^{t_2} \cdots u_{n-2m+1}^{t_m}  \\
    &=\underbrace{\underbrace{\teckaa{red}{u_{n-m}}u_{n-m}^{t_1-1}u_{n-m-1}^{t_2} \cdots {u_{n-2m+1}^{t_m}}}_{x} \teckaa{red}{}\cdots \teckaa{red}{u_{n-m}^{t_m}}}_{u_n}u_{n-m}^{t_1-t_m}u_{n-m-1}^{t_2} \cdots u_{n-2m+1}^{t_m}\,,
\end{split}
\end{equation}
where $x$ denotes the prefix of $z_n$ equal to $u_{n-m+1}$ for $n\geq 2m-1$ or equal to $u_{n-m+1}$ without the last letter for $n<2m-1$. The red positions correspond to $U_{n-m}-1$, $U_{n-m+1}-1$ and $U_{n}-1$. Using Lemma~\ref{lem: SP_rekurze}, we may write $x = u_{n-m}^{t_1}u_{n-m-1}^{t_2} \cdots u_{n-2m+1}^{t_m}$ and $x$ is thus also a~suffix of $z_n$. By Lemma~\ref{lem: SP_mocnina}, we can see that $x$ is a power of $u_{n-m}$. 
Now, every factor $f$ of $z_n$ crossing $U_{n-m}-1$, but not $U_{n-m+1}-1$, has an occurrence in $x$ containing $U_{n-m}-1$. Thanks to $t_1>t_m$, the first item of Observation~\ref{lem: atraktory_faktor} implies that $f$ has also an occurrence in $x$ containing $t_mU_{n-m}-1$. Since $x$ is, at the same time, a~suffix of $z_n$, the factor $f$ has an occurrence in $z_n$ containing $U_{n}-1$.
Therefore, $z_n$ has the attractor $\Gamma_{n} = \{U_{n-m+1}-1, \ldots, U_{n-1}-1, U_{n}-1\}$, too. See~\eqref{eq:z_n1}. 
Since $u_{n+1}$ is a power of $u_n$ and $U_n\leq Z_n<U_{n+1}$, we find by Lemma~\ref{lem: atraktory_mocnina} that every prefix of length $\ell\in [Z_n, U_{n+1}]$ has the attractor $\Gamma_n\cup\{U_n-1\}=\Gamma_n$.
\item For $t_1=t_m$, the proof of the fact that every prefix of length $\ell \in [S_n, U_{n+1}]$ has the attractor $\Gamma_{n}$ is analogous to the proof of the previous item. 
Consider now an arbitrary prefix of $\uu$ of length $\ell \in [Z_n, S_n]$. We want to show that $\Gamma=\Gamma_{n-1}\cup \{U_n-(t_m-1)U_{n-m}-1\}\setminus \{U_{n-m}-1\}$ is its attractor. 
Let us write the prefixes $z_n$ and $s_n$ below
\begin{equation}\label{eq:z_n}
\begin{split}
    z_n &= u_{n}u_{n-m}^{t_1-t_m}u_{n-m-1}^{t_2} \cdots u_{n-2m+1}^{t_m}  \\
    &= u_{n-1}^{t_1}u_{n-2}^{t_2} \cdots u_{n-m+1}^{t_{m-1}}u_{n-m}^{t_m}u_{n-m}^{t_1-t_m}u_{n-m-1}^{t_2} \cdots u_{n-2m+1}^{t_m}  \\
    &=\underbrace{\teckaa{red}{u_{n-m}}u_{n-m}^{t_1-1}u_{n-m-1}^{t_2} \cdots {u_{n-2m+1}^{t_m}}}_{x} \teckaa{red}{}\cdots \underbrace{\teckaa{red}{u_{n-m}}u_{n-m}^{t_1-1}u_{n-m-1}^{t_2} \cdots u_{n-2m+1}^{t_m}}_{x}\,,
\end{split}
\end{equation}
\begin{equation}\label{eq:s_n}
\begin{split}
    s_n &= u_{n}u_{n-m}^{t_1-t_m+1}u_{n-m-1}^{t_2} \cdots u_{n-2m+1}^{t_m}  \\
    &= u_{n-1}^{t_1}u_{n-2}^{t_2} \cdots u_{n-m+1}^{t_{m-1}}u_{n-m}^{t_m}u_{n-m}^{t_1-t_m+1}u_{n-m-1}^{t_2} \cdots u_{n-2m+1}^{t_m}  \\
    &=\underbrace{\teckaa{red}{u_{n-m}}u_{n-m}^{t_1-1}u_{n-m-1}^{t_2} \cdots \teckaa{red}{u_{n-2m+1}^{t_m}} \cdots \teckaa{red}{u_{n-m}}}_{y}u_{n-m}^{t_1}u_{n-m-1}^{t_2} \cdots u_{n-2m+1}^{t_m}\,,
\end{split}
\end{equation}
where, in both cases, the positions $U_{n-m}-1$, $U_{n-m+1}-1$ and $U_n-(t_m-1)U_{n-m}-1$ are in red.
Every factor $f$ of the prefix of length $\ell$ either crosses the last position $U_n-(t_m-1)U_{n-m}-1$ of $\Gamma$ or is contained in $u_n$ (more precisely, it is either contained in the prefix $y$, which is a~factor of $u_n$ or in the suffix $u_{n-m}^{t_1}u_{n-m-1}^{t_2} \cdots u_{n-2m+1}^{t_m}$, which is a prefix of $u_{n-m+1}$) and thus crosses $\Gamma_{n-1}$. It remains to explain that if $f$ is contained in $u_n$ and crosses $U_{n-m}-1$, but not $U_{n-m+1}-1$, then $f$ also crosses some other position of $\Gamma$. In such case, $f$ is contained in the prefix $x=\teckaa{red}{u_{n-m}}u_{n-m}^{t_1-1}u_{n-m-1}^{t_2} \cdots u_{n-2m+1}^{t_m}$ of $z_n$ and crosses the red position, thus $f$ is also contained in the suffix of $z_n$ and crosses the last position of $\Gamma$, see~\eqref{eq:z_n}. 

\end{enumerate}
In all cases, the considered prefixes contain all $m$ letters and the size of the attractors is equal to the alphabet size $m$. Consequently, the attractors are minimal. Furthermore, since $\Gamma_k \subset \{U_n-1 \ : \ n \in \mathbb N\}$ for all $k\in \mathbb N$, the attractors differ by at most one position from a canonical attractor.
\end{proof}

We have prepared everything for the description of minimal attractors of prefixes of simple Parry sequences. 
We start with the description for binary simple Parry sequences, where minimal attractors of all prefixes are canonical, i.e., are subsets of $\{U_n-1 \ : \ n \in \mathbb N\}$. 
For general alphabet size, we determine the minimal attractors of prefixes of simple Parry sequences in two theorems.
In the first one, the attractors of all prefixes are again canonical, but some additional conditions are imposed. In the second one, no additional condition is required, but some prefixes do not necessarily have canonical attractors (nevertheless, they differ by at most one element from a canonical attractor). In all cases, the minimality of attractors follows from the fact that they are of alphabet size.

\begin{proposition}\label{prop: binary_ternary}
Let $\uu$ be a binary simple Parry sequence from Definition~\ref{def:simple-Parry}. 

\begin{itemize} 
\item For $n \in \{0,1\}$, the prefix of $\uu$ of length $\ell \in [U_n, U_{n+1}-1]$ has the attractor $\Gamma_n$. 
\item For each $n \in \mathbb{N}$, $n \geq 2$, the prefix of $\uu$ of length $\ell \in [U_n, Z_n]$ has the attractor $\Gamma_{n-1}$.
\item For each $n \in \mathbb{N}$, $n \geq 2$, the prefix of $\uu$ of length $\ell \in [Z_n, U_{n+1}]$ has the attractor $\Gamma_n$.
\end{itemize}
Moreover, in all cases, the attractors are minimal. 
\end{proposition}

\begin{proof}
The first statement corresponds to Item 1 of Corollary~\ref{coro: thm10}. 

To prove the second statement, we will show that $z_n$ is a power of $u_{n-1}$.
Then by Item~4 of Corollary~\ref{coro: thm10}, the prefix of length $\ell \in [U_n, Z_n]$ has the attractor $\Gamma_{n-1}$. 
For $m=2$, the prefix $z_n$, for $n\geq 2$, has the following form
\begin{equation*}
\begin{split}
    z_n &= u_{n}u_{n-2}^{t_1-t_2}u_{n-3}^{t_2} \\
    &= u_{n-1}^{t_1}u_{n-2}^{t_2} u_{n-2}^{t_1-t_2}u_{n-3}^{t_2}\\
    &=u_{n-1}^{t_1}u_{n-2}^{t_1}u_{n-3}^{t_2}\,.
    \end{split} 
\end{equation*}
By Lemma~\ref{lem: SP_rekurze}, the word $u_{n-2}^{t_1}u_{n-3}^{t_2}$ is a prefix of $u_{n-1}$. Consequently, $z_n$ is a power of $u_{n-1}$.

To show the third statement, applying Proposition~\ref{prop: z}, it suffices to show that for $t_1=t_2$, any prefix of length $\ell \in [Z_n, S_n]$ has the attractor $\Gamma_n=\{U_{n-1}-1, U_n-1\}$.
The prefixes $s_n$ and $z_n$, for $n\geq 4$, have the following form
\begin{equation}\label{eq: s_n2}
\begin{split}
    s_n &= u_{n}u_{n-2}^{t_1-t_2+1}u_{n-3}^{t_2} \\
    &=u_{n-1}^{t_1}u_{n-2}^{t_1+1}u_{n-3}^{t_1}\\
    &=\underbrace{\underbrace{\teckaa{red}{u_{n-2}}u_{n-2}^{t_1-1}\teckaa{red}{u_{n-3}^{t_1}}}_{u_{n-1}} u_{n-1}^{t_1-1}\teckaa{red}{u_{n-2}^{t_1}}}_{u_n}u_{n-2}u_{n-3}^{t_1}\,. 
    \end{split}
\end{equation}
\begin{equation}\label{eq: z_n2}
\begin{split}
 z_n &=\teckaa{red}{u_{n-2}}u_{n-2}^{t_1-1}\underbrace{\teckaa{red}{u_{n-3}^{t_1}} u_{n-1}^{t_1-1}\teckaa{red}{u_{n-2}^{t_1}}}_{y}u_{n-3}^{t_1}\,. 
\end{split} 
\end{equation}
The positions $U_{n-2}-1, \ U_{n-1}-1, \ U_n-1$ are depicted in red in~\eqref{eq: s_n2} and~\eqref{eq: z_n2}.
Each factor $f$ of any prefix of length $\ell \in [Z_n, S_n]$ either crosses the last position $U_{n}-1$ of $\Gamma_n$ or is contained in $u_n$ and crosses $\Gamma_{n-1}=\{U_{n-2}-1, U_{n-1}-1\}$.
If the factor $f$ crosses the first red position $U_{n-2}-1$ and not the second one $U_{n-1}-1$ in the prefix $u_{n-1}=\teckaa{red}{u_{n-2}}u_{n-2}^{t_1-1}\teckaa{red}{u_{n-3}^{t_1}}$, then by Observation~\ref{lem: atraktory_faktor}, $f$ crosses $jU_{n-2}-1$ for all $1\leq j<t_1$. If $f$ moreover crosses $t_1U_{n-2}-1$, then $f$ crosses the last position $U_n-1$ of $\Gamma_n$. If $f$ does not cross $t_1U_{n-2}-1$, then $f$ is a suffix of the word $u_{n-2}^{t_1-1}x$, where $x$ is a prefix of $u_{n-2}$ longer than $u_{n-3}^{t_1}$, i.e., $U_{n-3}^{t_1}<|x|<U_{n-2}$. Since the factor $y$ underlined in~\eqref{eq: z_n2} has the prefix $u_{n-2}=u_{n-3}^{t_1}u_{n-4}^{t_1}$, the factor $f$ is contained in $u_{n-2}^{t_1-1}y$ and crosses the position $U_{n-1}-1$.

Let us finally check that for $n\in\{2,3\}$, each prefix of length $\ell \in [Z_n, S_n]$ has the attractor $\Gamma_n$, too.

For $n=2$, 
\begin{equation*}
\begin{array}{rcl}
    z_2&=&u_2=\underbrace{0^{t_1}\red{\textbf{1}}\cdots 0^{t_1}1}_{t_1 \times}0^{t_1-1}\red{\textbf{0}}\,,\\
    s_2&=&u_2u_0=\underbrace{0^{t_1}\red{\textbf{1}}\cdots 0^{t_1}1}_{t_1 \times}0^{t_1-1}\red{\textbf{0}}0\,.
\end{array}
\end{equation*}
One can easily check that $\Gamma_2=\{U_1-1, U_2-1\}$ (highlighted in red in $z_2$ and $s_2$) is clearly an attractor of both $z_2$ and $s_2$.

For $n=3$, the prefixes $z_3$ and $s_3$ have the following form
\begin{equation*}
\begin{split}
   z_3 &= u_{3}u_{0}^{t_1} \\   
   &= u_{2}^{t_1}u_{1}^{t_1}u_{0}^{t_1}\\
   &=\teckaa{red}{\underbrace{\teckaa{red}{\underbrace{\teckaa{red}{u_1}u_{1}^{t_1-1}u_{0}^{t_1}}_{u_2}}u_{2}^{t_1-1}u_{1}^{t_1}}_{u_3}}u_{0}^{t_1}\,,\\
   s_3&=\teckaa{red}{\underbrace{\teckaa{red}{\underbrace{\teckaa{red}{u_1}u_{1}^{t_1-1}u_{0}^{t_1}}_{u_2}}u_{2}^{t_1-1}u_{1}^{t_1}}_{u_3}}u_1u_{0}^{t_1}\,. 
   \end{split} 
\end{equation*}
As already shown, the prefix $z_3$ has the attractor $\Gamma_2=\{U_1-1, U_2-1\}$. Since $s_3$ is a power of $u_3$ and $U_3<Z_3<S_3$, by Lemma~\ref{lem: atraktory_mocnina}, $\{U_1-1, U_2-1, U_3-1\}$ is an attractor of every prefix of $\uu$ of length $\ell \in [Z_3, S_3]$ (the positions are highlighted in red in $z_3$ and $s_3$). Since $u_1=u_0^{t_1}\tt 1$, we can use Observation~\ref{lem: atraktory_faktor} with $x=u_1^{t_1}u_0^{t_1}=z^{t_1}z'$ with $z=u_1$ (its second item does not apply) and we can deduce that if a factor $f$ has an occurrence in $x$ crossing $|z|-1$, then for each $j$, $1\leq j\leq t_1$, the factor $f$ has an occurrence crossing $j|z|-1$. In other words, if a factor $f$ of the prefix of $\uu$ of length $\ell$ crosses $U_1-1$, but not $U_2-1$, then $f$ also crosses $U_3-1$. Therefore $\Gamma_3=\{U_2-1, U_3-1\}$ is an attractor of the prefix of length $\ell$, too.

It is readily seen that each attractor from Proposition~\ref{prop: binary_ternary} has the size equal to the number of distinct letters contained in the prefix. Consequently, the attractors are minimal. 
\end{proof}

\begin{example}
Let us illustrate the attractors from Proposition~\ref{prop: binary_ternary} on the prefixes of $\uu$ from Definition~\ref{def:simple-Parry}, where $m=2$ and $t_1=t_2=2$. Let us emphasize that Theorem~\ref{thm: affine} cannot be applied here since $t_2>1$. We choose several prefixes of $\uu$ and denote in red the positions of the attractor from Proposition~\ref{prop: binary_ternary}. Notice that $U_2=Z_2$, $|u_3u_0^2|=Z_3$ and $|u_4u_1^2|=Z_4$.
\begin{equation*}
\begin{array}{lcl}
u_0 &=& \red{\textbf{0}}\,, \\
u_1 &=& \red{\textbf{0}}0\red{\textbf{1}}\,, \\
z_2=u_2 &=& \red{\textbf{0}}0\red{\textbf{1}}00100\,, \\
z_2=u_2 &=& 00\red{\textbf{1}}0010\red{\textbf{0}}\,,\\
u_3 &=& 00\red{\textbf{1}}0010\red{\textbf{0}}00100100001001\,,\\
z_3 = u_3u_0^2 &=& 00\red{\textbf{1}}0010\red{\textbf{0}}0010010000100100\,,\\
z_3 = u_3u_0^2 &=& 0010010\red{\textbf{0}}0010010000100\red{\textbf{1}}00\,,\\
u_4 &=& 0010010\red{\textbf{0}}0010010000100\red{\textbf{1}}001001000010010000100100100100\\
&&00100100\,, \\
z_4 = u_4u_1^2 &=& 0010010\red{\textbf{0}}0010010000100\red{\textbf{1}}001001000010010000100100100100\\
&&00100100001001\,,\\
z_4 = u_4u_1^2 &=& 001001000010010000100\red{\textbf{1}}001001000010010000100100100100\\
&&0010010\red{\textbf{0}}001001\,.
\end{array}
\end{equation*}

\end{example}

Let us proceed to a~general alphabet size. First, we state a~theorem with assumptions distinct from Theorem~\ref{thm: affine} guaranteeing that prefixes have canonical minimal attractors.

\begin{theorem}\label{veta: SP_atraktory}
Let $\uu$ be a simple Parry sequence from Definition~\ref{def:simple-Parry}. Assume that
\begin{enumerate}
\item 
$t_i\cdots t_{m-2}(t_{m-1}+1)0^{\omega} \prec_{\text{lex}} t_1t_2 \cdots t_m 0^{\omega}$ for all $i \in \{2, \ldots, m-2\}$; 

\item $t_1 > \max\{t_{m-1}, t_m\}$.
\end{enumerate}
Then the prefixes of $\uu$ have the following attractors:
\begin{itemize} 
\item For each $n \in \mathbb{N}$, $n \leq m-1$, the prefix of $\uu$ of length $\ell \in [U_n, U_{n+1}-1]$ has the attractor $\Gamma_n$. 
\item For each $n \in \mathbb{N}$, $n \geq m$, the prefix of $\uu$ of length $\ell \in [U_n, Z_n]$ has the attractor $\Gamma_{n-1}$.
\item For each $n \in \mathbb{N}$, $n \geq m$, the prefix of $\uu$ of length $\ell \in [Z_n, U_{n+1}]$ has the attractor $\Gamma_n$.
\end{itemize}
Moreover, in all cases, the attractors are minimal. 
\end{theorem}

\begin{proof}
The first statement is a direct consequence of Corollary~\ref{coro: thm10}. The third statement, using the assumption $t_1>t_m$, follows from Item 1 of Proposition~\ref{prop: z}.
It remains to prove the second statement.
Consider the prefix $z_n$ of $\uu$, where $n\geq m$. We will show that $z_n$ is a power of $u_{n-1}$. Then by Item~4 of Corollary~\ref{coro: thm10}, the prefix of length $\ell \in [U_n, Z_n]$ has the attractor $\Gamma_{n-1}$.

It suffices to show that the prefix $z_n$ is a~power of $u_{n-1}$.
By Lemma~\ref{lem: SP_rekurze},
\begin{equation*}
\begin{split}
    z_n &= u_{n}u_{n-m}^{t_1-t_m}u_{n-m-1}^{t_2} \cdots  u_{n-2m+1}^{t_m} \\
    &= u_{n-1}^{t_1}u_{n-2}^{t_2} \cdots u_{n-m+1}^{t_{m-1}}u_{n-m}^{t_m}u_{n-m}^{t_1-t_m}u_{n-m-1}^{t_2} \cdots   u_{n-2m+1}^{t_m} \\
    &=u_{n-1}^{t_1}u_{n-2}^{t_2} \cdots u_{n-m+1}^{t_{m-1}}u_{n-m}^{t_1}u_{n-m-1}^{t_2} \cdots  u_{n-2m+1}^{t_m}\,.
\end{split} 
\end{equation*}
Using Lemma~\ref{lem: SP_rekurze}, the word $u_{n-m}^{t_1}u_{n-m-1}^{t_2} \cdots  u_{n-2m+1}^{t_m}$ is a~prefix of $u_{n-m+1}$. Consequently, $z_n$ is a~prefix of $u_{n-1}^{t_1}u_{n-2}^{t_2} \cdots u_{n-m+1}^{t_{m-1}+1}$. 
The lexicographic condition $t_i \cdots t_{m-2}(t_{m-1}+1)0^{\omega} \prec_{\text{lex}} t_1t_2 \cdots t_m0^{\omega}$ for all $i \in \{2, \ldots, m-2\}$ and $t_{m-1}<t_1$ implies that $u_{n-2}^{t_2} \cdots u_{n-m+1}^{t_{m-1}+1}$ is a~prefix of $u_{n-1}$ by Lemma~\ref{lem: SP_prefixy}. Thus $z_n$ is a~power of $u_{n-1}$.

It is easy to see that for each prefix, its attractor from Theorem~\ref{veta: SP_atraktory} has the size equal to the number of distinct letters contained in the prefix. Consequently, the attractors are minimal. 
\end{proof}
\begin{example}
Let us illustrate the attractors from Theorem~\ref{veta: SP_atraktory} on the prefixes of $\uu$ from Definition~\ref{def:simple-Parry}, where $m=3$ and $t_1=3$, $t_2=0$, $t_3=2$. Let us emphasize that Theorem~\ref{thm: affine} cannot be applied here since $t_3>1$. We choose several prefixes of $\uu$ and denote in red the positions of the attractor from Theorem~\ref{veta: SP_atraktory}. Notice that $|u_3u_0|=Z_3$ and $|u_4u_1|=Z_4$.

\begin{equation*}
\begin{array}{lcl}
u_0 &=& \red{\textbf{0}}\,,\\
u_1 &=& \red{\textbf{0}}00\red{\textbf{1}}\,,\\
u_2 &=& \red{\textbf{0}}00\red{\textbf{1}}00010001\red{\textbf{2}}\,,\\
u_3 &=& \red{\textbf{0}}00\red{\textbf{1}}00010001\red{\textbf{2}}0001000100012000100010001200\,,\\
z_3 = u_3u_0 &=& \red{\textbf{0}}00\red{\textbf{1}}00010001\red{\textbf{2}}00010001000120001000100012000\,,\\
z_3 = u_3u_0 &=& 000\red{\textbf{1}}00010001\red{\textbf{2}}000100010001200010001000120\red{\textbf{0}}0\,,\\
u_4 &=& 000\red{\textbf{1}}00010001\red{\textbf{2}}000100010001200010001000120\red{\textbf{0}}00010001000\\
&&1200010001000120001000100012000001000100012000100010\\
&&001200010001000120000010001\,, \\
z_4 = u_4u_1 &=& 000\red{\textbf{1}}00010001\red{\textbf{2}}000100010001200010001000120\red{\textbf{0}}00010001000\\
&&1200010001000120001000100012000001000100012000100010\\
&&0012000100010001200000100010001\,, \\
z_4 = u_4u_1 &=& 000100010001\red{\textbf{2}}000100010001200010001000120\red{\textbf{0}}00010001000\\
&&1200010001000120001000100012000001000100012000100010\\
&&00120001000100012000001000\red{\textbf{1}}0001\,.
\end{array}
\end{equation*}
\end{example}

\begin{example}\label{ex: atraktory2}
Here, we want to illustrate that the assumptions on the parameters $t_1, t_2,\dots, t_m$ from Theorem~\ref{veta: SP_atraktory} cannot be skipped. 
Consider $m=4$ and $t_1=2$, $t_2=1$, $t_3=2$ and $t_4=1$, i.e., the simplest case where neither assumptions of Theorem~\ref{thm: affine} nor assumptions of Theorem~\ref{veta: SP_atraktory} are met.

In this case, $z_6 = p_6 = u_6u_2u_1u_0^2$ and $P_6=U_6+12$ and $Z_6=U_6+13$. 

We will explain that the prefix $v$ of length $U_6+9 \in [U_6, P_6]\subset [U_6, Z_6]$ does not have the attractor $\Gamma_5$. The set $\Gamma_5$ is the attractor of $u_6$ and it is pointed out in red in the prefix $u_6$. It is easy to check that the underlined suffix $(0010010200100102001003)^2 00100102001001020$ of $v$ does not cross the set $\Gamma_5$.
The set $\Gamma_6$ is denoted in red in $v$. Again, it is not an attractor of $v$ since the underlined prefix of $v$ does not cross $\Gamma_6$.
It is not difficult to show that no minimal attractor of $v$ is canonical.
\begin{equation*}
\begin{array}{lcl}
u_{6}&=&0010010\red{\textbf{2}}0010010200100\red{\textbf{3}}001001020010010200100300100102001001\red{\textbf{0}}\\
&&00100102001001020010030010010200100102001003001001020010010\\
&&0010010200100102001003001001020010010200\red{\textbf{1}}001001020010010200\\
&&10030010010200100102001003001001020010010001001020010010200\\
&&10030010010200100102001003001001020010010001001020010010200\\
&&10030010010200100102001001001020010010200100300100102001001\\
&&02001003001001020010010001001020010010200100300100102001001\\
&&0200100300100102\,,\\
v&=&\underline{001001020010010200}100\red{\textbf{3}}001001020010010200100300100102001001\red{\textbf{0}}\\
&&00100102001001020010030010010200100102001003001001020010010\\
&&0010010200100102001003001001020010010200\red{\textbf{1}}001001020010010200\\
&&10030010010200100102001003001001020010010001001020010010200\\
&&10030010010200100102001003001001020010010001001020010010200\\
&&10030010010200100102001001001020010010200100300100102001001\\
&&02001003001001020010010\underline{001001020010010200100300100102001001}\\
&&\underline{020010030010010\red{\textbf{2}}001001020}\,.
\end{array}
\end{equation*}
\end{example}

In the following theorem, we introduce minimal attractors of prefixes of simple Parry sequences where no additional condition is imposed on the parameters.
\begin{theorem}\label{veta: SP_atraktory2}
Let $\uu$ be a simple Parry sequence from Definition~\ref{def:simple-Parry}. 
Denote 
\begin{equation*}
    k=\min\bigl\{j \in \{1,\dots, {m-1}\}\ : \ t_{m-j}\neq 0\bigr\}\,.
\end{equation*}
\begin{itemize}
\item
For each $n \in \mathbb{N}$, $n \leq m-1$, the prefix of $\uu$ of length $\ell \in [U_n, U_{n+1}-1]$ has the attractor $\Gamma_n$.
\item The prefix $u_m$ of length $\ell=U_m$ has the attractor $\Gamma_{m-1}$.
\item 
For each $n \in \mathbb{N}$, $n \geq m$, consider the prefix $u$ of $\uu$ of length $\ell \in [Z_n, U_{n+1}]$,
\begin{enumerate}
\item if $t_1>t_m$, then $u$ has the attractor $\Gamma_{n}$;
\item if $t_1=t_m$, then 
\begin{itemize}
\item if $u$ is of length $\ell \in [Z_n, S_n]$, then $u$ has the attractor 
\begin{equation*}
    \Gamma_{n-1}\cup \{U_n-(t_m-1)U_{n-m}-1\}\setminus \{U_{n-m}-1\}\,;
\end{equation*}
\item if $u$ is of length $\ell \in [S_n, U_{n+1}]$, then $u$ has the attractor $\Gamma_{n}$.
\end{itemize}
\end{enumerate}
\item 
For each $n \in \mathbb{N}$, $n \geq m$, consider the prefix $u$ of $\uu$ of length $\ell \in [U_n, Z_n]$, 
\begin{enumerate}
\item if $t_1\geq 2$, then $u$ has the attractor
\begin{equation*}
    \Gamma_{n-1}\cup\{U_n-U_{n-m+k}-(t_m-1)U_{n-m}-1\}\setminus\{U_{n-m}-1\}\,;
\end{equation*}
\item if $t_1=1$ and $t_{m-k}0^{k-1}t_m 0^{\omega} \prec_{lex}t_1\cdots t_{k+1} 0^{\omega}$, then $u$ has the attractor
\begin{equation*}
    \Gamma_{n-1}\cup\{U_n-U_{n-m+k}-1\}\setminus\{U_{n-m}-1\}\,;
\end{equation*}
\item if $t_1=1$ and $t_{m-k}0^{k-1}t_m =t_1\cdots t_{k+1}$, then
\begin{enumerate}
\item if $\ell\leq U_n+U_{n-m+k+1}-U_{n-m+k}-U_{n-m}$, then $u$ has the attractor
\begin{equation*}
    \Gamma_{n-1}\cup\{U_n-U_{n-m+k}-1\}\setminus\{U_{n-m}-1\}\,;
\end{equation*}
\item if $\ell\geq U_n+U_{n-m+k+1}-U_{n-m+k}-U_{n-m}$, then $u$ has the attractor
\begin{equation*}
    \Gamma_{n-1}\cup\{U_n-U_{n-m}-1\}\setminus \{U_{n-m+k}-1\}\,.
\end{equation*}
\end{enumerate}
\end{enumerate}


\end{itemize}

Moreover, in all cases, the attractors are minimal. 
\end{theorem}
\begin{proof}
The first two statements follow from Item~1 and Item~3 of Corollary~\ref{coro: thm10}.
The third one follows from Proposition~\ref{prop: z}.
Assume $n\geq m$. We want to confirm the form of attractors for every prefix $u$ of $\uu$ of length $\ell \in [U_n, Z_n]$.
Let us explain that every such prefix $u$ falls into one of the following two categories:
\begin{enumerate}
\item $u=u_nx$, where $u_{n-m+k}u_{n-m}^{t_m}x$ is a prefix of $u_{n-m+k+1}$,
\item $u=u_nx$, where $u_{n-m+k}u_{n-m}^{t_m}x$ has the prefix $u_{n-m+k+1}$.
\end{enumerate} 
Since $z_n=u_{n}u_{n-m}^{t_1-t_m}u_{n-m-1}^{t_2} \cdots u_{n-2m+1}^{t_m}$, the factor $u_{n-m}^{t_m}x$ is a prefix of $u_{n-m}^{t_1}u_{n-m-1}^{t_2} \cdots u_{n-2m+1}^{t_m}$, which is a prefix of $u_{n-m+1}$ by Lemma~\ref{lem: SP_rekurze}. To sum up, $u_{n-m+k}u_{n-m}^{t_m}x$ is a prefix of $u_{n-m+k}u_{n-m+1}$.

\begin{enumerate}
\item[a)] If $t_1\geq 2$ or $t_1=1$ and $t_{m-k}0^{k-1}t_m 0^{\omega} \prec_{lex}t_1\cdots t_{k+1} 0^{\omega}$, then $u_{n-m+k}u_{n-m+1}$ is a prefix of $u_{n-m+k+1}$ by Lemma~\ref{lem: SP_prefixy}. (Notice that under the assumption $t_1=1$, we have by the condition~\eqref{ParryRenyi} on the Rényi expansion of unity that $t_i\leq 1$ for all $i$ and the strict lexicographic inequality means that $k\geq 2$ and $10^{k-2}10^\omega\preceq_{lex}t_1\cdots t_{k+1} 0^{\omega}$.) Since $u_{n-m+k}u_{n-m}^{t_m}x$ is a prefix of $u_{n-m+k}u_{n-m+1}$, the word $u_{n-m+k}u_{n-m}^{t_m}x$ is a prefix of $u_{n-m+k+1}$ for every $\ell \in [U_n, Z_n]$.
\item[b)] If $t_1=1$ and $t_{m-k}0^{k-1}t_m =t_1\cdots t_{k+1}=10^{k-1}1$, then $u_{n-m+k}u_{n-m+1}$ has the prefix $u_{n-m+k+1}$. The explanation follows. By the condition on $t_1,\dots, t_m$, the form of $u_{n-m+k+1}$ reads
\begin{equation*}
    u_{n-m+k+1}=u_{n-m+k}u_{n-m}u_{n-m-1}^{t_{k+2}}\cdots u_{n-2m+k+1}^{t_m}\,.
\end{equation*}
By Lemma~\ref{lem: SP_prefixy}, the word $u_{n-m}u_{n-m-1}^{t_{k+2}}\cdots u_{n-2m+k+1}^{t_m}$ is a prefix of $u_{n-m+1}$, therefore indeed $u_{n-m+k+1}$ is a prefix of $u_{n-m+k}u_{n-m+1}$. 
Consequently, in this last case, there exists $L \in [U_n, Z_n]$ such that $u_{n-m+k}u_{n-m}^{t_m}x=u_{n-m+k}u_{n-m}x$ is a prefix of $u_{n-m+k+1}$ for all $\ell\leq L$ and $u_{n-m+k}u_{n-m}x$ has the prefix $u_{n-m+k+1}$ for all $\ell\geq L$. The reader can easily verify that $L=U_n+U_{n-m+k+1}-U_{n-m+k}-U_{n-m}$.
\end{enumerate}

\begin{enumerate}
\item Assume $u=u_nx$, where $u_{n-m+k}u_{n-m}^{t_m}x$ is a prefix of $u_{n-m+k+1}$. We will prove that $u$ has the attractor 
\begin{equation*}
    \Gamma=\Gamma_{n-1}\cup\{U_n-U_{n-m+k}-(t_m-1)U_{n-m}-1\}\setminus\{U_{n-m}-1\}\,.
\end{equation*}
In particular, if $t_1=1$, then $t_m=1$ and the attractor simplifies to 
\begin{equation*}
    \Gamma=\Gamma_{n-1}\cup\{U_n-U_{n-m+k}-1\}\setminus\{U_{n-m}-1\}\,.
\end{equation*}
Let us express the prefix $u=u_nx$ in a~handy form (on the third line we use the definition of $k$)

\begin{equation}\label{eq:z2}
\begin{split}
     u&=u_{n}x\\
    &= u_{n-1}^{t_1}u_{n-2}^{t_2} \cdots u_{n-m+1}^{t_{m-1}}u_{n-m}^{t_m}x\\
    &=\underbrace{\underbrace{\underbrace{\teckaa{red}{u_{n-m}}u_{n-m}^{t_1-1}u_{n-m-1}^{t_2} \cdots \teckaa{red}{u_{n-2m+1}^{t_m}}}_{u_{n-m+1}}\teckaa{red}{\cdots}}_{u_{n-1}} u_{n-1}^{t_1-1}u_{n-2}^{t_2}\cdots u_{n-m+k}^{t_{m-k}} {u_{n-m}^{t_m}}}_{u_n} x\\
    &=\underbrace{\teckaa{red}{u_{n-m+1}}\teckaa{red}{\cdots} }_{u_{n-1}}u_{n-1}^{t_1-1}\cdots u_{n-m+k}^{t_{m-k}-1}\underbrace{\teckaa{red}{u_{n-m}}u_{n-m}^{t_1-1}u_{n-m-1}^{t_2}\cdots u_{n-2m+1}^{t_m}\cdots}_{u_{n-m+k}}u_{n-m}^{t_m}x\,.
\end{split}
\end{equation}
The prefix $u_n$ has the attractor $\Gamma_{n-1}=\{U_{n-m}-1, U_{n-m+1}-1, \dots, U_{n-1}-1 \}$ by Item 2 of Corollary~\ref{coro: thm10}, which is highlighted in red on the penultimate line. 
We will explain that the prefix $u=u_nx$ of $z_n$ has the attractor $\Gamma$ that is obtained from $\Gamma_{n-1}$ by leaving out the position $U_{n-m}-1$ and adding the position $U_n-U_{n-m+k}-(t_m-1)U_{n-m}-1$ ($\Gamma$ is denoted in red on the last line of~\eqref{eq:z2}):
If $f$ is a~factor of the suffix $u_{n-m+k}u_{n-m}^{t_m}x$ of $u_nx$, then $f$ is, by assumption, a~factor of $u_{n-m+k+1}$, i.e., $f$ occurs in $u_n$.
It follows that every factor $f$ of the prefix $u$ is either contained in the prefix $u_n$ and crosses $\Gamma_{n-1}$ or crosses the last position of $\Gamma$, i.e., the position $U_n-U_{n-m+k}-(t_m-1)U_{n-m}-1$.
Moreover, every factor of $u_n$ that crosses the position $U_{n-m}-1$ and not $U_{n-m+1}-1$ has also an occurrence containing the position $U_n-U_{n-m+k}-(t_m-1)U_{n-m}-1$, i.e., the last position of $\Gamma$ (see the penultimate and the last line of~\eqref{eq:z2}).

\item Assume $u=u_nx$, where $u_{n-m+k}u_{n-m}^{t_m}x$ has the prefix $u_{n-m+k+1}$. We have already inspected that this happens only for $t_1=1$ and $t_{m-k}0^{k-1}t_m =t_1\cdots t_{k+1}=10^{k-1}1$. We will prove that $u$ has the attractor 
\begin{equation*}
    \Gamma=\Gamma_{n-1}\cup\{U_n-U_{n-m}-1\}\setminus \{U_{n-m+k}-1\}\,.
\end{equation*}
Let us express the prefix $u$ in another handy form (on the second line, we simplify the expression using the knowledge $t_1=t_m=1$, on the third line, using the definition of $k$)
\begin{equation}\label{eq:z3}
\begin{split}
    u &=u_{n}x\\
    &= u_{n-1}u_{n-2}^{t_2} \cdots u_{n-m+1}^{t_{m-1}}u_{n-m}x \\
    &=\underbrace{\underbrace{\underbrace{\underbrace{\teckaa{red}{u_{n-m}} \cdots \teckaa{red}{}}_{u_{n-m+k}}\cdots \teckaa{red}{}}_{u_{n-m+k+1}}\teckaa{red}{\cdots}}_{u_{n-1}} u_{n-2}^{t_2}\cdots u_{n-m+k} {u_{n-m}}}_{u_n} x\\
    &=\underbrace{\underbrace{\underbrace{\underbrace{\teckaa{red}{u_{n-m}} \cdots {}}_{u_{n-m+k}}\cdots \teckaa{red}{}}_{u_{n-m+k+1}}\teckaa{red}{\cdots}}_{u_{n-1}} u_{n-2}^{t_2}\cdots \teckaa{red}{u_{n-m+k}} {u_{n-m}}}_{u_n} x\,.
\end{split}
\end{equation}
The prefix $u_n$ has the attractor $\Gamma_{n-1}=\{U_{n-m}-1, U_{n-m+1}-1, \dots, U_{n-1}-1 \}$ by Item 2 of Corollary~\ref{coro: thm10}, which is highlighted in red on the penultimate line. 
We will explain that the prefix $u=u_nx$ of $z_n$ has the attractor $\Gamma$ that is obtained from $\Gamma_{n-1}$ by leaving out the position $U_{n-m+k}-1$ and adding the position $U_n-U_{n-m}-1$ ($\Gamma$ is denoted in red on the last line of \eqref{eq:z3}):
If $f$ is a~factor of the suffix $u_{n-m}x$, then $f$ is a factor of $u_{n-m+1}$ by our very first observation, hence $f$ is a factor of $u_{n}$.
It follows that every factor $f$ of the prefix $u$ is either contained in the prefix $u_n$ and crosses $\Gamma_{n-1}$ or crosses the last position of $\Gamma$, i.e., the position $U_n-U_{n-m}-1$.
Moreover, every factor of $u_n$ that crosses the position $U_{n-m+k}-1$ and not $U_{n-m+k+1}-1$ is contained in $u_{n-m+k+1}$, therefore $f$ is also contained in $u_{n-m+k}u_{n-m}x$ and crosses the position $U_n-U_{n-m}-1$, i.e., the last position of $\Gamma$ (see~\eqref{eq:z3}).
\end{enumerate}

Let us point out that for each prefix, its attractor from Theorem~\ref{veta: SP_atraktory2} has the size equal to the number of distinct letters contained in the prefix. Consequently, the attractors are minimal.\end{proof}



\begin{example}
Let us illustrate the attractors of prefixes of $\uu$ from Example~\ref{ex: atraktory2}, where $m=4$ and $t_1=2$, $t_2=1$, $t_3=2$ and $t_4=1$. Recall that neither assumptions of Theorem~\ref{thm: affine} nor assumptions of Theorem~\ref{veta: SP_atraktory} are met. We apply Theorem~\ref{veta: SP_atraktory2}. 
The attractors of prefixes from Theorem~\ref{veta: SP_atraktory2} are highlighted in red. For the prefixes of length smaller than $U_4$, the attractors from Theorem~\ref{veta: SP_atraktory} and Theorem~\ref{veta: SP_atraktory2} coincide. The prefixes $u_n$ and $z_n$, for $n\geq 4$, have two different attractors by Theorem~\ref{veta: SP_atraktory2}. 

The length of the factor $v$ from Example~\ref{ex: atraktory2} satisfies $|v|\in [U_6, Z_6]$. By the proof of Theorem~\ref{veta: SP_atraktory2}, as $t_1\geq 2$, the attractor of $v$ equals $\Gamma=\{U_3-1, U_4-1, U_5-1, U_6-U_3-1\}$; see the picture below.
Let us repeat the argument why $\Gamma$ is indeed an attractor of $v$. 
Each factor $f$ of $v$ either crosses the last position of $\Gamma$ or is contained in $u_6$ and crosses $\Gamma_5$. If $f$ occurs in the prefix $x=0010010\red{\textbf{2}}0010010200100$ of length $U_3-1$ and crosses the red position $U_2-1$, then $f$ clearly has an occurrence in $v$ containing the last position of $\Gamma$. We underline both above-mentioned occurrences of $x$.
\begin{equation*}
\begin{array}{lcl}
u_0 &=& \red{\textbf{0}}\,, \\
u_1 &=& \red{\textbf{0}}0\red{\textbf{1}}\,, \\
u_2 &=& \red{\textbf{0}}0\red{\textbf{1}}0010\red{\textbf{2}}\,, \\
v_1 &=& \red{\textbf{0}}0\red{\textbf{1}}0010\red{\textbf{2}}00100\,, \\
u_3 &=& \red{\textbf{0}}0\red{\textbf{1}}0010\red{\textbf{2}}0010010200100\red{\textbf{3}}\,, \\
v_2 &=& \red{\textbf{0}}0\red{\textbf{1}}0010\red{\textbf{2}}0010010200100\red{\textbf{3}}001001020\,,\\
u_4 &=& \red{\textbf{0}}0\red{\textbf{1}}0010\red{\textbf{2}}0010010200100\red{\textbf{3}}0010010200100102001003001001020010010\,, \\
u_4 &=& {0}0\red{\textbf{1}}0010\red{\textbf{2}}0010010200100\red{\textbf{3}}001001020010010200100300100102001\red{\textbf{0}}010\,, \\
z_4=u_4u_0 &=& {0}0\red{\textbf{1}}0010\red{\textbf{2}}0010010200100\red{\textbf{3}}001001020010010200100300100102001\red{\textbf{0}}0100\,, \\
z_4=u_4u_0 &=& {0}0\red{\textbf{1}}0010\red{\textbf{2}}0010010200100\red{\textbf{3}}001001020010010200100300100102001{0}01\red{\textbf{0}}0\,, \\
v_3 &=& 00\red{\textbf{1}}0010\red{\textbf{2}}0010010200100\red{\textbf{3}}001001020010010200100300100102001001\red{\textbf{0}}0\\
&&010010200100102001003001001020010\,,\\
u_5 &=& 00\red{\textbf{1}}0010\red{\textbf{2}}0010010200100\red{\textbf{3}}001001020010010200100300100102001001\red{\textbf{0}}\\
&&00100102001001020010030010010200100102001003001001020010010\\
&&0010010200100102001003001001020010010200{1}\,,\\
u_5 &=& 00{1}0010\red{\textbf{2}}0010010200100\red{\textbf{3}}001001020010010200100300100102001001\red{\textbf{0}}\\
&&00100102001001020010030010010200100102001003001001020010010\\
&&00100102001001020010030010010200\red{\textbf{1}}0010200{1}\,,\\
z_5 =u_5u_1u_0&=& 00{1}0010\red{\textbf{2}}0010010200100\red{\textbf{3}}001001020010010200100300100102001001\red{\textbf{0}}\\
&&00100102001001020010030010010200100102001003001001020010010\\
&&00100102001001020010030010010200\red{\textbf{1}}0010200{1}0010\,,\\
z_5=u_5u_1u_0 &=& 00{1}0010\red{\textbf{2}}0010010200100\red{\textbf{3}}001001020010010200100300100102001001\red{\textbf{0}}\\
&&00100102001001020010030010010200100102001003001001020010010\\
&&00100102001001020010030010010200{1}0010200\red{\textbf{1}}0010\,,\\
v_{5}&=&0010010\red{\textbf{2}}0010010200100\red{\textbf{3}}001001020010010200100300100102001001\red{\textbf{0}}\\
&&00100102001001020010030010010200100102001003001001020010010\\
&&0010010200100102001003001001020010010200\red{\textbf{1}}001001020010010200\\
&&1003001001020010010200100300100102001001000100102001\,,\\
u_{6}&=&0010010\red{\textbf{2}}0010010200100\red{\textbf{3}}001001020010010200100300100102001001\red{\textbf{0}}\\
&&00100102001001020010030010010200100102001003001001020010010\\
&&0010010200100102001003001001020010010200\red{\textbf{1}}001001020010010200\\
&&10030010010200100102001003001001020010010001001020010010200\\
&&10030010010200100102001003001001020010010001001020010010200\\
&&10030010010200100102001001001020010010200100300100102001001\\
&&02001003001001020010010001001020010010200100300100102001001\\
&&0200100300100102\,,\\
u_{6}&=&0010010{2}0010010200100\red{\textbf{3}}001001020010010200100300100102001001\red{\textbf{0}}\\
&&00100102001001020010030010010200100102001003001001020010010\\
&&0010010200100102001003001001020010010200\red{\textbf{1}}001001020010010200\\
&&10030010010200100102001003001001020010010001001020010010200\\
&&10030010010200100102001003001001020010010001001020010010200\\
&&10030010010200100102001001001020010010200100300100102001001\\
&&0200100300100102001001000100102001001020010030010010\red{\textbf{2}}001001\\
&&0200100300100102\,,\\
\end{array}
\end{equation*}
\begin{equation*}
\begin{array}{rcl}
v&=&\underline{001001020010010200100}\red{\textbf{3}}001001020010010200100300100102001001\red{\textbf{0}}\\
&&00100102001001020010030010010200100102001003001001020010010\\
&&0010010200100102001003001001020010010200\red{\textbf{1}}001001020010010200\\
&&10030010010200100102001003001001020010010001001020010010200\\
&&10030010010200100102001003001001020010010001001020010010200\\
&&10030010010200100102001001001020010010200100300100102001001\\
&&020010030010010200100100010010200100102001003\underline{0010010\red{\textbf{2}}001001}\\
&&\underline{0200100}300100102001001020\,.
\end{array}
\end{equation*}
\end{example}

\begin{example}
Let us illustrate the attractors of prefixes of the simple Parry sequence with parameters $m=5$ and $t_1=t_2=1$, $t_3=0$, $t_4=t_5=1$. We can use Theorem~\ref{veta: SP_atraktory2} with $k=1$. The prefix $u_{9}$ has the attractor $\Gamma_{8}$; the positions of $\Gamma_{8}$ are depicted in red in $u_{9}$. Consider the prefixes $v_1, v_2\in [U_{9}, Z_{9}]=[U_{9}, U_{9}+10]$. 

Since $v_1=u_90102$ and $|v_1|=U_9+4\leq U_9+6=U_9+U_6-U_5-U_4$, by Theorem~\ref{veta: SP_atraktory2}, the prefix $v_1$ has the attractor $\Gamma=\{U_5-1, U_6-1, U_7-1, U_8-1, U_{9}-U_{5}-1\}$; again highlighted in red in $v_1$.
Let us repeat the argument why $\Gamma$ is indeed an attractor of $v_1$. 
Each factor $f$ of $v_1$ either crosses the last position of $\Gamma$ or is contained in $u_{9}$ (the underlined suffix of $v_1$ is equal to the underlined prefix of $v_1$) and crosses $\Gamma_{8}$. If $f$ occurs in the prefix $z=010201301020\red{4}010201301$ (depicted in blue) of length $U_5-1$ and crosses the red position $U_4-1$, then $f$ clearly has an occurrence in $v_1$ containing the last position of $\Gamma$ (the corresponding occurrence of $z$ is also depicted in blue).

Since $v_2=u_9010201301$ and $|v_2|=U_9+9>U_9+6=U_9+U_6-U_5-U_4$, by Theorem~\ref{veta: SP_atraktory2}, the prefix $v_2$ has the attractor $\widehat\Gamma=\{U_4-1, U_6-1, U_7-1, U_8-1, U_{9}-U_{4}-1\}$; again denoted in red in $v_2$.
Let us repeat the argument why $\widehat\Gamma$ is indeed an attractor of $v_2$. 
Each factor $f$ of $v_2$ either crosses the last position of $\widehat\Gamma$ or is contained in $u_{9}$ (the suffix $w$ of $v_2$ is at the same time a~prefix of $v_2$) and crosses $\Gamma_{8}$. If $f$ occurs in the prefix of length $U_6-1$, i.e., in $y=0102013010204010201301\red{0}010201301020401020$ (depicted in gray), and crosses the red position $U_5-1$, then $f$ has an occurrence in $v_2$ containing the last position of $\widehat\Gamma$ (the corresponding occurrence of $y$ is also depicted in gray).

The assumptions of Theorems~\ref{thm: affine} and~\ref{veta: SP_atraktory} are not satisfied. On the one hand, the prefix $v_1$ has the attractor $\Gamma_8$:  $v_1$ is a~power of $u_8$ and $u_8$ has the attractor $\Gamma_7$. Consequently, $v_1$ has the attractor $\Gamma_7 \cup \{U_8-1\}=\{U_3-1\}\cup \Gamma_8$. Every factor $f$ of $v_1$ crossing the position $U_3-1$ but not $U_4-1$ also crosses $U_8-1$. On the other hand, $v_2$ does not have the attractor $\Gamma_8$. The reader is invited to check that the suffix $y1301=010201301020401020130100102013010204010201301$ of $v_2$ does not cross any position of $\Gamma_{8}$. 

\begin{equation*}
\begin{array}{lcl}
u_4&=&0102013010204\\
u_5&=&01020130102040102013010\\
u_6&=&010201301020401020130100102013010204010201\\
u_7&=&0102013010204010201301001020130102040102010102013010204010201301001020130102\\
u_{8}&=&0102013010204010201301001020130102040102010102013010204010201301001020130102\\
&&01020130102040102013010010201301020401020101020130102040102013\\
u_{9}&=&010201301020\red{\textbf{4}}010201301\red{\textbf{0}}010201301020401020\red{\textbf{1}}010201301020401020130100102013010\red{\textbf{2}}\\
&&0102013010204010201301001020130102040102010102013010204010201\red{\textbf{3}}01020130102040\\
&&1020130100102013010204010201010201301020401020130100102013010201020130102040\\
&&1020130100102013010204\\
v_1&=&\underline{\textcolor{blue}{0102013010204010201301}\red{\textbf{0}}01020130102040102}0\red{\textbf{1}}010201301020401020130100102013010\red{\textbf{2}}\\
&&0102013010204010201301001020130102040102010102013010204010201\red{\textbf{3}}01020130102040\\
&&10201301001020130102040102010102013010204010201301001020130102\underline{\textcolor{blue}{010201301020}\red{\textbf{4}}\textcolor{blue}{0}}\\
&&\underline{\textcolor{blue}{10201301}001020130102040102}\\
v_2&=&\underbrace{\textcolor{gray}{010201301020}\red{\textbf{4}}\textcolor{gray}{010201301}}_w\textcolor{gray}{0010201301020401020}\red{\textbf{1}}010201301020401020130100102013010\red{\textbf{2}}\\
&&0102013010204010201301001020130102040102010102013010204010201\red{\textbf{3}}01020130102040\\
&&10201301001020130102040102010102013010204010201301001020130102\textcolor{gray}{01020130102040}\\
&&\textcolor{gray}{10201301}\red{\textbf{0}}\underbrace{\textcolor{gray}{010201301020401020}1301}_w
\end{array}
\end{equation*}

\end{example}

\section{Attractors of prefixes of binary non-simple Parry sequences}\label{sec:nonsimpleParry}
Gheeraert, Romana, and Stipulanti~\cite{GhRoSt2024} mentioned as an open problem finding minimal attractors of prefixes of non-simple Parry sequences. In this section, we answer their question for prefixes of the form $\varphi^n(0)$ of binary non-simple Parry sequences. 

Let us recall the definition of binary non-simple Parry sequences in the form of fixed points of morphisms, the assumptions on parameters follow from the properties of the Rényi expansion of unity~\eqref{ParryRenyi}.
\begin{definition}\label{def:non-simple-Parry} A~binary \emph{non-simple Parry sequence} $\uu$ is a~fixed point of the morphism $\varphi: \{0,1\}^*\to \{0,1\}^*$ defined as
\begin{equation*}
\begin{array}{rcl}
    \varphi(0)&=&0^{p}1, \\
    \varphi(1)&=&0^{q}1,
\end{array}
\end{equation*}
where $p,q\in \mathbb{N}$, $p> q\geq 1$.
\end{definition}

\begin{example}\label{pr: NSP_priklad}

For $p=3$, $q=1$, the morphism $\varphi$ is defined as
\begin{equation*}
\begin{array}{rclcl}
\varphi(0) &=& 0001, \\
\varphi(1) &=& 01,
\end{array}
\end{equation*}
and the first five prefixes $\varphi^n(0)$ of $\uu$ look as follows
\begin{equation*}
\begin{array}{rcl}
\varphi^0(0)&=&0\,,\\
\varphi^1(0) &=& 0001\,, \\
\varphi^2(0) &=& 00010001000101\,, \\
\varphi^3(0) &=& 000100010001010001000100010100010001000101000101\,, \\
\varphi^4(0) &=& 00010001000101000100010001010001000100010100010100010001000101000100010001\\
&&01000100010001010001010001000100010100010001000101000100010001010001010001\\
&&0001000101000101\,.
\end{array}
\end{equation*}

\end{example}
\begin{remark}
It is known that $\uu$ is Sturmian if and only if $p=q+1$.
Attractors of prefixes of Sturmian sequences~\cite{Mantaci2021} are known.
\end{remark}

\begin{remark}
In the non-simple Parry case, it does not make sense to search for canonical attractors, i.e., attractors being subsets of $\{|\varphi^n(0)|-1 \ : \ n\in \mathbb N\}$. No attractor of a prefix containing both letters can form a~subset of $\{|\varphi^n(0)|-1 \ : \ n\in \mathbb N\}$, as happened in the simple Parry case.
The reason is that $\varphi^n(0)$ always ends in $1$ for $n\geq 1$, hence the positions $|\varphi^n(0)|-1$ for $n\geq 1$ are, without exception, occurrences of the letter $1$.
\end{remark}

All statements of the next handy lemma can be proved by induction.
\begin{lemma}\label{lem: NSP_prefix}
The following statements hold for the morphism $\varphi$ from Definition~\ref{def:non-simple-Parry}.
\begin{enumerate}
\item $\varphi^{n+1}(0)=\left(\varphi^n(0)\right)^p\varphi^n(1)$;
\item $\varphi^k(1)$ is a suffix of $\varphi^k(0)$ for each $k\in \mathbb N, k\geq 1$;
\item $\varphi^k(0)\varphi^{k-1}(0) \cdots \varphi(0)0$ is a prefix of $\varphi^{k+1}(0)$ for each $k \in \mathbb{N}$;
\item $1\varphi(1) \cdots \varphi^{k-1}(1)\varphi^k(1)$ is a suffix of $\varphi^{k+1}(0)$ for each $k \in \mathbb{N}$;
\item the prefix from Item 3 and the suffix from Item 4 do not overlap in $\varphi^{k+1}(0)$;
\item $\varphi^k(0)\varphi^{k-1}(0) \cdots \varphi(0)0$ is a prefix of $\varphi^{k+1}(1)$ for each $k\in \mathbb N$;
\item $\varphi(1)\varphi^2(1) \cdots \varphi^k(1)$ is a suffix of $\varphi^k(0)$ for each $k \in \mathbb{N},\ k\geq 1$;
\item $\varphi(1)\varphi^2(1)\cdots \varphi^k(1)\varphi^k(0)\varphi^{k-1}(0)\cdots \varphi(0)$ is a factor of $\varphi^{k+1}(0)$ for each $k\in \mathbb N, k\geq 1$.
\end{enumerate}
\end{lemma}

Now, we can prove the theorem on minimal attractors of prefixes $\varphi^n(0)$ of binary non-simple Parry sequences.

\begin{theorem}\label{veta: NSP_atraktory}
Let $\uu$ be a binary non-simple Parry sequence from Definition~\ref{def:non-simple-Parry}. For each $n \in \mathbb{N},\ n\geq 1,$ the prefix $\varphi^n(0)$ has the minimal attractor 
\begin{equation*}
    \Gamma_n =  \left\{\sum_{j=0}^{n-1} |\varphi^j(0)|-1,\quad |\varphi^n(0)|-\sum_{j=1}^{n-1} |\varphi^j(1)|-1\right \}\,.
\end{equation*}
\end{theorem}

\begin{proof}
For $n \in \mathbb{N}, \ n\geq 1,$ by Item 3 of Lemma~\ref{lem: NSP_prefix}, the word $\varphi^{n-1}(0)\varphi^{n-2}(0) \cdots \varphi(0)0$ is a prefix of $\varphi^n(0)$, and by Item 4 of Lemma~\ref{lem: NSP_prefix}, the word $1\varphi(1) \cdots \varphi^{n-2}(1)\varphi^{n-1}(1)$ is a suffix of $\varphi^n(0)$. 
Moreover, by Item 5, they do not overlap. Consequently, $\varphi^n(0)$ has the form
\begin{equation}\label{eq:un}
   \varphi^n(0) = \varphi^{n-1}(0)\varphi^{n-2}(0) \cdots \varphi(0)\red{\textbf{0}} \cdots \red{\textbf{1}}\varphi(1) \cdots \varphi^{n-2}(1)\varphi^{n-1}(1).
\end{equation}
For $n \in \mathbb{N}, \ n\geq 1,$ we will show by induction that $\varphi^n(0)$ has the attractor 
\begin{equation*}
    \Gamma_n = \left \{\sum_{j=0}^{n-1} |\varphi^j(0)|-1,\quad |\varphi^n(0)|-\sum_{j=1}^{n-1} |\varphi^j(1)|-1\right \}\,;
\end{equation*}
the positions of the attractor are highlighted in red in~\eqref{eq:un}.

For $n = 1$, the prefix $\varphi(0)=0^{p}1$ clearly has the attractor $\Gamma_1 = \{|0|-1, |\varphi(0)|-1\} = \{0, p\}$. The positions of the attractor $\Gamma_1$ are denoted below in red
\begin{equation*}
   \varphi(0) = \underbrace{\red{\textbf{0}}0 \cdots 0}_{p\text{-times}}\red{\textbf{1}}\,.
\end{equation*}

Let us assume that the statement holds for some $n\geq 1$, i.e., $\varphi^n(0)$ has the attractor 
\begin{equation*}
    \Gamma_n =  \left\{\sum_{j=0}^{n-1} |\varphi^j(0)|-1,\quad |\varphi^n(0)|-\sum_{j=1}^{n-1} |\varphi^j(1)|-1\right\}\,.
\end{equation*}
We will show that $\varphi^{n+1}(0)$ has the attractor 
\begin{equation*}
    \Gamma_{n+1} = \left \{\sum_{j=0}^n |\varphi^j(0)|-1,\quad |\varphi^{n+1}(0)|-\sum_{j=1}^n |\varphi^j(1)|-1\right\}\,;
\end{equation*}
depicted below in red. The prefix $\varphi^{n+1}(0)$ has the following form, where $u=\left(\varphi^n(0)\right)^p$ by Item 1 of Lemma~\ref{lem: NSP_prefix},
\begin{equation}\label{eq: non-simple}
    \varphi^{n+1}(0)= \underbrace{\varphi^n(0)\underbrace{\varphi^{n-1}(0) \cdots \varphi(0)\red{\textbf{0}} \cdots}_{\varphi^n(0)} \cdots\underbrace{\cdots\red{\textbf{1}}\varphi(1) \cdots \varphi^{n-1}(1)}_{\varphi^n(0)}}_u\varphi^n(1)\,.
\end{equation}
Each factor $f$ of $\varphi^{n+1}(0)$ has either an occurrence containing the position $|\varphi^{n+1}(0)|-\sum_{j=1}^n |\varphi^j(1)|-1$ (corresponding to the red letter $1$) or $f$ is a factor of $u$ or $f$ is a factor of $\varphi(1) \cdots \varphi^{n-1}(1)\varphi^n(1)$, which is a suffix of $\varphi^n(0)$ by Item 7 of Lemma~\ref{lem: NSP_prefix}, thus $f$ is again a factor of $u$. 
Using the fact that $u$ is a power of $\varphi^n(0)$, if a factor $f$ of $u$ is of length greater than or equal to $|\varphi^n(0)|$, then $f$ necessarily crosses the position $\sum_{j=0}^{n} |\varphi^j(0)|-1$ (corresponding to the red letter $0$). 

It remains to consider a factor $f$ of $u$ shorter than $\varphi^n(0)$.
\begin{itemize}
\item If $f$ is contained in $\varphi^n(0)$, then $f$ crosses by induction assumption the attractor $\Gamma_n$ in $\varphi^n(0)$, hence $f$ clearly crosses the attractor $\Gamma_{n+1}$ in $\varphi^{n+1}(0)$. Compare~\eqref{eq: non-simple} and~\eqref{eq:un}.
\item Assume now that $f$ is not a factor of $\varphi^n(0)$. We will show by contradiction that $f$ crosses $\Gamma_{n+1}$. 
Since $f$ is not contained in $\varphi^n(0)$, necessarily, $f$ has an occurrence containing the two middle positions of $\varphi^n(0)\varphi^n(0)$. 
\begin{equation}\label{eq:middle}
    \varphi^n(0)|\varphi^n(0)= \cdots \textcolor{blue}{\textbf{1}}\varphi(1) \cdots \varphi^{n-1}(1)|\varphi^{n-1}(0)\varphi^{n-2}(0) \cdots \varphi(0)\textcolor{blue}{\textbf{0}}\cdots
\end{equation}

Moreover, as we suppose that $f$ does not cross $\Gamma_{n+1}$, the considered occurrence of $f$ does not contain the blue positions depicted in \eqref{eq:middle}. Let us explain why. If $f$ contains the blue $\textcolor{blue}{\textbf{0}}$, then $f$ clearly crosses the position $\sum_{j=0}^{n} |\varphi^j(0)|-1$ (corresponding to the red letter $0$) in the attractor $\Gamma_{n+1}$. Assume $f$ does not contain the blue $\textcolor{blue}{\textbf{0}}$, but contains the blue $\textcolor{blue}{\textbf{1}}$, then by Item 6 of Lemma~\ref{lem: NSP_prefix},  $f$ is contained in $\varphi^n(0)\varphi^n(1)$ and crosses the position $|\varphi^{n+1}(0)|-\sum_{j=1}^n |\varphi^j(1)|-1$ (corresponding to the red letter $1$) in the attractor $\Gamma_{n+1}$.

For $n=1$, we have $\varphi(0)|\varphi(0)=0^p\textcolor{blue}{\textbf{1}}|\textcolor{blue}{\textbf{0}}0^{p-1}1$, therefore such $f$ does not exists.
For $n\geq 2$, $f$ is a factor of $\varphi(1)\varphi^2(1) \cdots \varphi^{n-1}(1)\varphi^{n-1}(0)\varphi^{n-2}(0) \cdots \varphi(0)$, which is, by Item~8 of Lemma~\ref{lem: NSP_prefix}, a~factor of $\varphi^n(0)$. This contradicts the assumption that $f$ is not contained in $\varphi^n(0)$.
\end{itemize}
To sum up, we have shown that each factor of $\varphi^{n+1}(0)$ crosses $\Gamma_{n+1}$.

Obviously, for each prefix, its attractor from Theorem~\ref{veta: NSP_atraktory} has the size equal to two, that is to the number of distinct letters contained in the prefix. Consequently, the attractors are minimal.
\end{proof}

\begin{example}
Let us illustrate the attractors from Theorem~\ref{veta: NSP_atraktory} on the prefixes $\varphi^n(0)$ from Example~\ref{pr: NSP_priklad}; the positions of attractors are highlighted in red. 
\begin{equation*}
\begin{array}{rcl}
\varphi^1(0) &=& \red{\textbf{0}}00\red{\textbf{1}}\,, \\
\varphi^2(0) &=& 0001\red{\textbf{0}}001000\red{\textbf{1}}01\,, \\
\varphi^3(0) &=& 000100010001010001\red{\textbf{0}}00100010100010001000\red{\textbf{1}}01000101\,, \\
\varphi^4(0)&=& 000100010001010001000100010100010001000101000101000100010001010001\red{\textbf{0}}0010001\\
&&0100010001000101000101000100010001010001000100010100010001000\red{\textbf{1}}010001010001\\
&&0001000101000101\,.
\end{array}
\end{equation*}

\end{example}

\section{Open problems}
Our research was inspired by the paper~\cite{GhRoSt2024}, where the authors studied attractors of prefixes of fixed points of morphisms of the form
\begin{equation}\label{c}
    0 \to 0^{c_0}1, \ 1 \to 0^{c_1}2, \ 2 \to 0^{c_2}3, \dots, m-1 \to 0^{c_{m-1}}, 
\end{equation}
where $c_i \in \mathbb N$ for all $i \in \{0, 1, \ldots, m-1\}$, $c_0 \geq 1, c_{m-1} \geq 1$.

Simple Parry sequences form a subclass of such fixed points. The authors found attractors of prefixes of size $m+1$, i.e., number of letters increased by one. However, they conjectured that attractors of alphabet size should exist. Furthermore, they asked under which conditions the minimal attractors are canonical, i.e., form a~subset of $\{U_n-1 \ : \ n \in \mathbb{N}\}$. 

In this paper, we proved that prefixes of simple Parry sequences indeed have attractors of alphabet size, i.e., we described minimal attractors of prefixes of simple Parry sequences, see Theorem~\ref{veta: SP_atraktory2}. Moreover, for binary sequences, see Proposition~\ref{prop: binary_ternary}, and for general sequences under some additional conditions, see Theorem~\ref{veta: SP_atraktory}, the minimal attractors we found are canonical. The assumptions of Theorem~\ref{thm: affine} and Theorem~\ref{veta: SP_atraktory} are sufficient, not necessary, therefore, the description of simple Parry sequences whose minimal attractors of all prefixes are canonical is not complete. 

In addition, the authors of~\cite{GhRoSt2024} asked how the minimal attractors of prefixes of non-simple Parry sequences look like. In this paper, we answered the question only for prefixes of some particular form in the binary case.

As mentioned, simple Parry sequences form a subclass of fixed points of morphisms from~\eqref{c}, hence it remains an open problem to find minimal attractors in full generality.
Concerning non-simple Parry sequences over larger alphabets, according to our brief experience, finding minimal attractors of prefixes seems to be a harder task than the simple Parry case.


Recently, Rényi numeration systems have been extended to the broader class of real Cantor numeration systems, where numbers are represented using a sequence of bases
$\mathcal{B} = (\beta_n)_{n \in \mathbb{Z}}$ with each $\beta_n > 1$. 
A~positive real number is then represented in the form
\begin{equation*}
    \sum_{n=0}^{N-1} a_n \beta_{n-1} \cdots \beta_0 + \sum_{n=1}^{\infty} \frac{a_{-n}}{\beta_{-1} \cdots \beta_{-n}}, \quad \text{where } N \geq 1, a_i \in \mathbb{N}, a_i < \beta_i \,.
\end{equation*}
The Rényi numeration system corresponds to the special case, where the base $\mathcal{B}$ is constant. 
A natural analogue of $\beta$-integers, the so-called $\mathcal{B}$-integers, can be defined in this setting.  As shown in the recent work~\cite{ChCiMaPe2025}, Cantor real numeration systems define a substantially larger class of sequences than the Rényi case. 
When the base $\mathcal{B}$ is taken to be a purely periodic sequence, we call it an
{\em alternate} base. 
In such case, the sequence coding distances between consecutive $\mathcal{B}$-integers is fixed by a primitive morphism. 
The attractors of prefixes / factors of such sequences are completely unknown.

In a broader context, it remains an open question to determine minimal attractors of prefixes / factors of fixed points of morphisms. The first steps in this direction have been done by Cassaigne et al.~\cite{Ca2024}.

\acknowledgements
We would like to thank two anonymous reviewers. In particular, one of them read the text in an exceptionally careful way and suggested many useful remarks that helped us make the proofs more readable and comprehensible.

\bibliographystyle{abbrv}
\bibliography{biblio}

\begin{thebibliography}{10}

\bibitem{AmFrMaPe2006}
P.~Ambrož, C.~Frougny, Z.~Masáková, and E.~Pelantová.
\newblock Palindromic complexity of infinite words associated to simple {Parry} numbers.
\newblock {\em Annales de l'Institut Fourier}, 56:2131--2160, 2006.

\bibitem{BaKlPe2011}
L.~Balková, K.~Klouda, and E.~Pelantová.
\newblock Critical exponent of infinite words coding beta-integers associated with non-simple {Parry} numbers.
\newblock {\em Integers Elec. J. Combin. Number Theory}, 11B:1--25, 2011.

\bibitem{BaMa2009}
L.~Balková and Z.~Masáková.
\newblock Palindromic complexity of infinite words associated with non-simple {Parry} numbers.
\newblock {\em RAIRO Theor. Inform. Appl.}, 43:145--163, 2009.

\bibitem{BeMaPe2007}
J.~Bernat, Z.~Masáková, and E.~Pelantová.
\newblock On a class of infinite words with affine factor complexity.
\newblock {\em Theoret. Comput. Sci.}, 389(1-2):12--25, 2007.

\bibitem{BeCrRo2025}
M.-P. Béal, M.~Crochemore, and G.~Romana.
\newblock Checking and producing word attractors.
\newblock arXiv:2509.08503, 2025.

\bibitem{BeGhMe2024}
P.~Béaur, F.~Gheeraert, and B.~{Hellouin de Menibus}.
\newblock String attractors and bi-infinite words.
\newblock arXiv:2403.13449, 2024.

\bibitem{Ca2024}
J.~Cassaigne, F.~Gheeraert, A.~Restivo, G.~Romana, M.~Sciortino, and M.~Stipulanti.
\newblock New string attractor-based complexities for infinite words.
\newblock {\em J. Combin. Theory Ser. A}, 208:105936, 2024.

\bibitem{ChCiMaPe2025}
{\'E}.~Charlier, C.~Cisternino, Z.~Mas\'akov\'a, and E.~Pelantov\'a.
\newblock Substitutions and {C}antor real numeration systems.
\newblock {\em Acta Math. Hungar.}, 176(1):15--47, 2025.

\bibitem{Dolce2023}
F.~Dolce.
\newblock String attractors for factors of the {Thue}-{Morse} word.
\newblock In A.~Frid and R.~Mercaş, editors, {\em Combinatorics on Words. WORDS 2023}, volume 13899 of {\em Lecture Notes in Computer Science}, pages 117--129. Springer, Cham., 2023.

\bibitem{DoMa2015}
D.~Dombek, Z.~Masáková, and T.~Vávra.
\newblock Confluent {Parry} numbers, their spectra, and integers in positive- and negative-base number systems.
\newblock {\em J. Th\'{e}orie Nombres Bordeaux}, 27:745--768, 2015.

\bibitem{Dv2023}
L.~Dvořáková.
\newblock String attractors of episturmian sequences.
\newblock {\em Theoret. Comput. Sci.}, 986:114341, 2024.

\bibitem{DvHe2024}
L.~Dvořáková and V.~Hendrychová.
\newblock String attractors of {Rote} sequences.
\newblock {\em Discrete Mathematics \& Theoretical Computer Science}, vol. 26:3, Nov 2024.

\bibitem{Fa1995}
S.~Fabre.
\newblock Substitutions et $\beta$-systèmes de numération.
\newblock {\em Theoret. Comput. Sci.}, 137:219--236, 1995.

\bibitem{GhRoSt2024}
F.~Gheeraert, G.~Romana, and M.~Stipulanti.
\newblock String attractors of some simple-{P}arry automatic sequences.
\newblock {\em Theory Comput. Syst.}, 68(6):1601–1621, 2024.

\bibitem{KempaPrezza2018}
D.~Kempa and N.~Prezza.
\newblock At the roots of dictionary compression: String attractors.
\newblock In {\em STOC 2018}, pages 827--840. ACM, 2018.

\bibitem{KlPe2009}
K.~Klouda and E.~Pelantová.
\newblock Factor complexity of infinite words associated with non-simple {Parry} numbers.
\newblock {\em Integers Elec. J. Combin. Number Theory}, 9:281--310, 2009.

\bibitem{Kociumaka2021}
T.~Kociumaka, G.~Navarro, and N.~Prezza.
\newblock Towards a definitive measure of repetitiveness.
\newblock In {\em {LATIN}, LNCS}, volume 12118, pages 207--219. Springer, 2020.

\bibitem{Kutsukake2020}
K.~Kutsukake, T.~Matsumoto, Y.~Nakashima, S.~Inenaga, H.~Bannai, and M.~Takeda.
\newblock On repetitiveness measures of {Thue}-{Morse} words.
\newblock In {\em {SPIRE}, LNCS}, volume 12303, pages 213--220. Springer, 2020.

\bibitem{Lothaire}
M.~Lothaire.
\newblock {\em Algebraic combinatorics on words}.
\newblock Cambridge University Press, Cambridge, 2002.

\bibitem{Mantaci2021}
S.~Mantaci, A.~Restivo, G.~Romana, G.~Rosone, and M.~Sciortino.
\newblock A combinatorial view on string attractors.
\newblock {\em Theoret. Comput. Sci.}, 850:236--248, 2021.

\bibitem{MaPe2011}
Z.~Masáková and E.~Pelantová.
\newblock Ito-{Sadahiro} numbers vs. {Parry} numbers.
\newblock {\em Acta Polytechnica}, 51:59--64, 2011.

\bibitem{Pa1960}
W.~Parry.
\newblock On the $\beta$-expansions of real numbers.
\newblock {\em Acta Math. Acad. Sci. Hungar.}, 11:401--416, 1960.

\bibitem{Re1957}
A.~Rényi.
\newblock Representations for real numbers and their ergodic properties.
\newblock {\em Acta Math. Acad. Sci. Hungar.}, 8:477--493, 1957.

\bibitem{Shallit2021}
L.~Schaeffer and J.~Shallit.
\newblock String attractors of automatic sequences.
\newblock arXiv:2012.06840, 2020.

\bibitem{Th1989}
W.~P. Thurston.
\newblock Groups, tilings, and finite state automata.
\newblock Research report gcg1, University of Minnesota, Geometry supercomputer project, 1989.

\bibitem{Tu2015}
O.~Turek.
\newblock Abelian properties of {Parry} words.
\newblock {\em Theoret. Comput. Sci.}, 566:26--38, 2015.

\end{thebibliography}
\label{sec:biblio}

\end{document}